\documentclass[10pt]{article}
\usepackage{amsfonts}
\usepackage{amsfonts}  

\usepackage{amsthm}
\usepackage{amsmath}
\usepackage{color}

\title{Ergodic pairs for singular or degenerate fully nonlinear operators} 

\author{I. Birindelli,  F. Demengel, F. Leoni}

\date{}

\catcode`@=11 \@addtoreset{equation}{section} \catcode`@=12

\newtheorem{theo}{Theorem}[section]
\newtheorem{prop}[theo]{Proposition}
\newtheorem{rema}[theo]{Remark}
\newtheorem{defi}[theo]{Definition}

\def\R{\mathbb  R}
\def\grad{\nabla}

\begin{document}
\maketitle
\begin{abstract}
We study the ergodic problem for fully nonlinear operators which may be singular or degenerate when the gradient of solutions vanishes. We prove the convergence of both explosive solutions and solutions of Dirichlet problems for  approximating equations. We further characterize the ergodic constant as the infimum of constants for which there exist bounded sub solutions. As intermediate results of independent interest, we prove a priori Lipschitz estimates depending only on the norm of the zeroth order term, and a comparison principle for equations having no zero order terms.
\end{abstract}
\section{Introduction} \label{intro}
In 1989, in a fundamental paper \cite{LL}, Lasry and Lions study solutions of 
$$
-\Delta u+|\grad u|^q+\lambda u=f(x)\  \mbox{in}\ \Omega$$
 that blow up on the boundary of $\Omega$. Here $q>1$ and $\Omega$ 
is a  ${ \cal C}^2$  bounded domain in  $\R^N$.
In particular they introduce the concept of ergodic pair.  Among other things, they prove that in the subquadratic case $q\leq 2$ there
exists a unique constant $c_\Omega$, called ergodic constant, and there exists a unique, up to a constant, solution
of
 $$
-\Delta \varphi+|\grad \varphi|^q-c_\Omega=f(x)\  \mbox{in}\ \Omega, \ \varphi=+\infty \ \mbox{on} \ \partial\Omega.$$
It is well known that for $q=2$, $-c_\Omega$ is just the principal eigenvalue of $(-\Delta+ f )(\cdot)$.
This important paper has generated a huge and interesting literature, also in connection with the stochastic 
interpretation of the problem.
Interestingly, while the concept of principal eigenvalue has been extended to fully nonlinear operators of 
different types (see e.g. \cite{BEQ}, \cite{BD1}), the notion of ergodic constant has not been much investigated in fully nonlinear 
settings. The scope of this paper is to give a thoroughly picture of the ergodic pairs and the related 
blowing up solutions and solutions with Dirichlet boundary condition for approximating equations.

We now detail the  main theorems.
In the whole paper $\Omega$ denotes a ${\cal C}^2$  bounded domain of  $\R^N$;  ${\cal S}$ denotes the space of symmetric matrices in $\R^N$.  We consider  a uniformly elliptic homogenous operator $F$, i.e. a continuous function $F:{\cal S}\to \R$ satisfying:
 \begin{equation}\label{unifel}
  \begin{array}{c}
  \hbox{ there exist  $0<a< A$ such that for all $M, N \in  {\cal S}$, with  $N\geq 0$, and for all $t>0$,}\\[1.5ex]
      a\,  {\rm  tr}(N)  \leq F( M+ N) -F(M) \leq A \,   {\rm  tr}(N),\  \ F(tM) = t F(M) .
\end{array}      \end{equation} 
We will always consider the differential operator 
$$G(\grad u, D^2u)=-|\grad u|^\alpha F(D^2u)+|\grad u|^\beta$$ 
with $\alpha>-1$ and $\alpha+1<\beta\leq \alpha+2.$
This reduces to the Lasry Lions case for $\alpha=0$ and $a=A=1$.
\begin{theo}\label{sympaetpas} Suppose that $f$ is  bounded and locally Lipschitz continuous  in $\Omega$, and that $F$ satisfies \eqref{unifel}. Consider the
Dirichlet problems 
             \begin{equation}\label{eq3}
             \left\{ \begin{array}{cl}
                     -| \nabla u |^\alpha F(D^2 u) + |\nabla u |^\beta   = f & \hbox{ in } \ \Omega \\
                     u=0 & \hbox{ on } \ \partial \Omega ,
                     \end{array}\right.
                     \end{equation}
and,  for $\lambda >0$, 
\begin{equation}\label{eq45}
\left\{ \begin{array}{cl}
                     -| \nabla u |^\alpha F(D^2 u) + |\nabla u |^\beta   + \lambda |u|^\alpha u  = f & \hbox{ in } \ \Omega \\
                     u=0 &  \hbox{ on } \ \partial \Omega .
                     \end{array}\right.
                     \end{equation}
The following alternative holds.
\begin{enumerate}
\item\label{sympa} Suppose that there exists a bounded sub solution of \eqref{eq3}.
 Then the solution $u_\lambda$ of  \eqref{eq45}                      
satisfies: $(u_\lambda)$ is bounded and uniformly converging up to a sequence $\lambda_n\to 0$ to  a  solution of (\ref{eq3}).
\item\label{passympa} Suppose that there is  no solution for the Dirichlet problem \eqref{eq3}.
 Then, $(u_\lambda)$  satisfies, up to a sequence $\lambda_n\to 0$ and locally uniformly in $\Omega$,
\begin{enumerate}
\item\label{a}  $u_\lambda\rightarrow -\infty$; 
\item\label{b} there exists a constant $c_\Omega\geq 0$ such that $\lambda |u_\lambda|^\alpha u_\lambda\to -c_\Omega;$
\item \label{c}  $c_\Omega $ is an ergodic constant  and  $v_\lambda=u_\lambda+|u_\lambda|_\infty$ converges  to  a solution of the ergodic problem    
\begin{equation}\label{ergodic}
\left\{ \begin{array}{cl}
                     -| \nabla v |^\alpha F(D^2 v) + |\nabla v |^\beta   = f +  c_\Omega  & \hbox{ in} \ \Omega \\
                     v=+\infty &  \hbox{ on} \ \partial \Omega 
                     \end{array}\right.
\end{equation}
whose minimum is zero.
\end{enumerate}
\end{enumerate}
\end{theo}
The standard notion of viscosity solution fails when the operator is singular, i.e., in this paper, when $\alpha<0$, 
hence we will consider viscosity solutions as defined in \cite{BD1}.

Theorem \ref{sympaetpas} gives a construction of an ergodic pair 
$(c_\Omega, v)$ as a limit of solutions of Dirichlet problems when problem \eqref{eq3} does not admit any solution. 
More in general, we will prove the existence of ergodic pairs under the regularity condition 
\begin{equation}\label{smooth}
F(\nabla d(x)\otimes \nabla d(x)) \hbox{ is $\mathcal{C}^2$ in a neighborhood of } \partial \Omega\, ,
\end{equation}
where $d(x)$ denotes here the distance function from $\partial \Omega$. 
We observe that \eqref{smooth} is certainly satisfied if the domain $\Omega$ is of class ${\cal C}^3$ and the operator $F$ is  ${\cal C}^2$, but there can be also cases with non smooth $F$ satisfying \eqref{smooth}. For instance, for all operators  $F(M)$ which depend only on the eigenvalues of $M$, such  as   Pucci's operators,  $F(\nabla d(x)\otimes \nabla d(x))$ is a constant function as long as $|\nabla d(x)|=1$. 

Under condition \eqref{smooth}, we will show that there exists a unique ergodic constant $c_\Omega$, which shares some properties with the principal eigenvalue
of the operator $|\grad u|^\alpha F(D^2u)$, even when it is not the principal eigenvalue.

Indeed,  let us define, as e.g. in \cite{BPT}, \cite{P},
          $$\mu^\star =  \inf \{ \mu\, :\,  \exists\,  \varphi \in { \cal C} ( \overline{ \Omega}), -| \nabla \varphi |^\alpha F( D^2 \varphi) + | \nabla \varphi |^\beta \leq f+ \mu \}\, .$$
Note that  $\mu^\star$ depends on $f$ and $\Omega$, but if there is no ambiguity we will not precise this dependence. 
 \begin{theo} \label{ABCDE}
Suppose that $f$ is  bounded and locally Lipschitz continuous  in $\Omega$, and that $F$ satisfies \eqref{unifel} and \eqref{smooth}. Then, there exists an ergodic constant $c_\Omega$ and it satisfies
           \begin{enumerate}
            \item\label{A}  $c_\Omega$ is unique; 
                       \item\label{B}   $c_\Omega = \mu^\star$; 
 \item\label{C}  the map $\Omega\mapsto c_\Omega$ is  nondecreasing with respect to the  domain, and continuous;
\item\label{D}   if  either $\alpha = 0$   or  $\alpha \neq 0$ and $\sup_\Omega f+ c_\Omega  <0$,  then  $\mu^\star$ is not achieved. 
Moreover,  if $ \Omega' \subset \subset \Omega$, then  $c_{\Omega' }< c_{ \Omega}$.  
\end{enumerate}
              \end{theo}
In order to prove these results many  questions need to be addressed. 
Clearly the first one is the existence of solutions for \eqref{eq45}
when $\lambda>0$, but even though  it is fundamental, this is extraneous
to the spirit of this note and it can be found in \cite{BDL1}. The interested reader will see
that it is done through a Perron's procedure i.e. constructing sub and super solutions of \eqref{eq45}
together with a comparison principle and some Lipschitz estimates depending on the $L^\infty$ norm of the solution.
\smallskip

Theorems \ref{sympaetpas} and \ref{ABCDE} are obtained by means of several  intermediate  results, most of which are of independent 
interest.
A first fundamental tool is an interior Lipschitz estimate for solutions of equation \eqref{eq45} that does not depend directly on the $L^\infty$ norm 
of the solution but only on the norm of the zero order term.
In the linear case, these kind of estimates were obtained by Capuzzo Dolcetta, Leoni, Porretta  in \cite{CDLP}, 
and the proof we use is inspired by theirs. In order to extend the result to the present fully nonlinear singular case, we have to address
several non trivial  technical difficulties, see  Section \ref{secLipLeoni}.
After that, we give the proof of Theorem \ref{sympaetpas} in Section \ref{theo1}.

In Section \ref{exiergolambda} we focus on  existence and estimates for explosive solutions of the approximating $\lambda$--equation, i.e.  solutions $u$ of
$$
\left\{ \begin{array}{cl}
                     -| \nabla u |^\alpha F(D^2 u) + |\nabla u |^\beta   + \lambda |u|^\alpha u  = f & {\rm in} \ \Omega \\
                     u=+\infty  & { \rm on} \ \partial \Omega.
                     \end{array}\right.
$$
Here, the function $f$ is assumed to
be continuous in $\Omega$, but it is allowed to be unbounded on the boundary, as long as its growth is controlled.  This is an 
important feature that will be needed in the proof of Theorem \ref{ABCDE}.  
Construction of explosive solutions in the fully nonlinear setting includes the works by Alarc\'on, Quaas \cite{AQ}, by 
Esteban, Felmer and Quaas \cite{EFQ} and by Demengel, Goubet \cite{DG}, where only suitable zero order terms 
are considered. Capuzzo Dolcetta, Leoni and Vitolo in \cite{CDLV1, CDLV2} construct radial explosive solutions in 
some degenerate cases.
The construction we do in order to give existence and
estimates  of blowing up solutions slightly differs from the standard proof for linear operators (see Remark \ref{new}),  and 
we obtain solutions satisfying  non constant boundary asymptotics. Moreover, our proof can be carried on 
for other classes of operators, as e.g. the p--Laplacian or some generalizations such as
$$F(p, M) = |p|^\alpha (q_1 tr M + q_2 \frac{M   p\cdot p}{|p|^2})\, ,$$
 with  $q_1>0$ and $q_1+q_2>0$. 
When $p>2$, using the variational form of the p--Laplacian, and its linearity with respect to the Hessian, 
Leonori and Porretta proved in \cite{LP} such estimates and existence results. So our result for $\alpha<0$ covers the case $p<2$ that was not considered there.  

Uniqueness of solution is a sensitive matter for degenerate elliptic equations, nonetheless, when $\alpha\geq 0$ i.e. in the degenerate case, uniqueness of explosive solutions is proved for any $\beta$. Instead, in the singular case, i.e. when $\alpha<0$, we have some restriction on $\beta$ and on the behavior of $f$ near the boundary.
In a forthcoming paper, \cite{BDL2},  we prove  some $W^{2,p}$ estimates for any $p$ when $f$ is continuous   in the  case $\alpha \leq 0$, which give  ${ \cal C}^{1, \eta}$ estimates.   This provides the uniqueness of blowing up solutions  when $\alpha \leq 0$, without any restriction on $\beta$ and  $f$.

\medskip
As we have seen when comparing the ergodic constant with the principal eigenvalue, in some sense the forcing term 
$f$ in \eqref{eq3} plays the role of the zero order term for linear problems, hence it is not surprising that for $f<0$ the 
comparison principle holds  even if the operator is degenerate and the equation is not proper. This is the spirit
of the comparison principle that we prove in Section \ref{compzero}. The change of equation that allows to prove the 
comparison principle of Theorem \ref{2.4} is sort of standard, but the computation which follows  is original 
and ad hoc for our setting. The work of Leonori, Porretta and Riey \cite{LPR} has been a source of inspiration.
  
Finally,  in Section  \ref{sectionerg}, after proving the existence of ergodic constants and  estimating the ergodic functions near the boundary, we
 complete the proof of 
Theorem  \ref{ABCDE}. 
\bigskip

{\bf Notations} 
\begin{itemize}
\item We use $d(x)$ to denote a  ${\cal C}^2$ positive function in $\Omega$ with coincides with the distance function from the boundary in a neighborhood of 
$\partial \Omega$

\item For $\delta>0$, we  set $\Omega_\delta = \{ x\in \Omega\, :\, d(x)>\delta \}$

\item We denote by $\mathcal{M}^+, \mathcal{M}^-$ the Pucci's operators with ellipticity constants $a, A$, namely, for all $M\in {\cal S}$,
$$
\begin{array}{c}
\mathcal{M}^+(M)= A\, {\rm tr}(M^+)-a\, {\rm tr}(M^-)\\[1ex]
\mathcal{M}^-(M)= a\, {\rm tr}(M^+)-A\, {\rm tr}(M^-)
\end{array}
$$
and we often use that, as a consequence of \eqref{unifel}, for all $M, N\in {\cal S}$ one has
$$
\mathcal{M}^-(N) \leq F(M+N)-F(M)\leq \mathcal{M}^+(N)
$$
\end{itemize}

 \section{A priori Lipschitz-type estimates }\label{secLipLeoni}

In the note \cite{BDL1}, we prove the following result

\begin{theo}\label{existence} Assume that $f$ is bounded and continuous in $\Omega$. Then, for any $\lambda>0$, there
exists a unique  solution $u_\lambda \in C(\overline \Omega)$ of \eqref{eq45}, which is  Lipschitz continuous up to the boundary, and satisfies 
$$|u_\lambda|_{W^{1, \infty} ( \Omega)} \leq C( |u_\lambda|_\infty,  |f-\lambda |u_\lambda|^\alpha u_\lambda |_\infty, a, A,\alpha, \beta ).$$ 
\end{theo}
This is obtained, through Perron's method, by constructing sub and super solutions   and using the following general comparison principle.

\begin{theo}\label{comp} Suppose that $b$ is Lipschitz continuous in $\Omega$, 
$\zeta:\R \rightarrow\R $ is a non decreasing function and  $f$ and $g$ are continuous in  $\Omega$.
Let $u$ be a bounded by above viscosity sub solution of 
      $$-|\grad u|^\alpha F(D^2 u)+ b(x)|\nabla u|^{\beta}+ \zeta (u) \leq  g$$
and let $v$ be a  bounded by below viscosity super solution of 
       $$-|\grad v|^\alpha F(D^2 v)+ b(x)|\nabla v|^{\beta} + \zeta (v)\geq  f.$$
If either $g \leq f$ and $\zeta $ is increasing or $g < f$ then,
 $u\leq v$ on $\partial\Omega$  implies that 
      $u \leq v$ in $\Omega$. 
\end{theo}

For the proofs of the above results we refer to \cite{BDL1}. 

The rest of this section is devoted to prove  a priori Lipschitz estimates 
 for  solutions of the equation
\begin{equation}\label{eqq1}
-|\nabla u|^\alpha F(D^2 u) + |\nabla u|^{\beta} + \lambda |u|^\alpha u = f \, ,
\end{equation}
that depend on $\lambda |u|_\infty^{\alpha +1}$, but  not on $|u|_\infty$.
Our estimates will be a consequence of the following  result, in which we denote by $B$ the 
unit ball centred at the origin in $\R^N$.
 
  \begin{prop}\label{lipunif}  Let $F$ satisfy \eqref{unifel} and, for $\lambda\geq 0, \alpha>-1$ and $\beta>\alpha+1$, let $u$ and $v$ be respectively a bounded  sub solution and a bounded from below   super solution of equation \eqref{eqq1} in $B$, with $f$ Lipschitz continuous in $B$. Then, for any positive $p\geq \frac{(2+\alpha-\beta)^+}{\beta-\alpha-1}$, there exists a positive constant $M$, depending only on $p, \alpha, \beta, a, A, N, \|f-\lambda |u|^\alpha u\|_\infty$ and on the Lipschitz constant of $f$, such that, for all $x, y\in B$ one has
  $$
 u(x)-v(y)\leq \sup_B(u-v)^+ + M \frac{|x-y|}{(1-|y|)^{\frac{\beta+\alpha^-}{\beta-\alpha-1}}} \left[ 1+\left( \frac{|x-y|}{(1-|x|)}\right)^p\right]
 $$
 \end{prop}
\begin{proof}
We argue as in the proof of Theorem 3.1  in \cite{CDLP}. 
     
Let us define a "distance" function $d$ which  equals $1-|x|$ near the boundary and it is extended in $B$ as a $\mathcal{C}^2$
function   satisfying,  for some constant $c_1>0$,
      $$\left\{ \begin{array}{cc}
      d(x) = 1-|x| & {\rm if}  \ |x| > \frac{1}{2}\\[1ex]
        \frac{1-|x|}{ 2} \leq d(x) \leq 1-|x| & {\rm  for\  all} \ x\in \bar B  \\[1ex]
         |Dd (x) | \leq 1\, ,\ -c_1 I_N \leq  D^2 d (x) \leq 0 & {\rm for \ all} \ x\in \bar B
         \end{array} \right.$$
 For   $\xi =\frac {|x-y|}{d(x)}$,  we consider  the function 
     $$\phi(x, y) = \frac{k}{d(y)^\tau} |x-y| \left(L+ \xi^p \right)+\sup_B(u-v)^+ $$
where $p>0$ is a fixed exponent satisfying $p\geq \frac{(2+\alpha-\beta)^+}{\beta-\alpha-1}$, $\tau=\frac{\beta+\alpha^-}{\beta-\alpha-1}$ and $L$, $k$ are suitably large positive constants to be  chosen  later.
   
 The statement is proved if we show that  for all $(x,y) \in  B^2$ one has
      $$ u(x)-u(y) \leq \phi(x, y)\, .$$
    
By contradiction, let us assume that  $u(x)-u(y)-\phi(x,y)  >0$ somewhere.  Then,  necessarily
          the supremum is  achieved on a pair $(x,y)$ with  $x \neq y$ and  $d(x)\, , d(y) >0$. Using Ishii's Lemma of \cite{I1}, one gets that on such a point $(x,y)$,  for all $\epsilon>0$, there exist symmetric matrices $X_\epsilon $ and $Y_\epsilon $  such that 
 \begin{equation} \label{eqine}
\begin{array}{c}
\displaystyle  (\grad_x  \phi, X_\epsilon ) \in J^{2,+} u( x), \ (-\grad_y \phi, -Y_\epsilon ) \in J^{2,-} v(y)\\[2ex]
\displaystyle -\left( \frac{1}{\epsilon } + |D^2 \phi|\right) I_{2N} \leq \left(\begin{array}{cc} 
X_\epsilon   & O\\
O & Y_\epsilon \end{array} \right) \leq D^2\phi + \epsilon  (D^2 \phi)^2.
\end{array}
 \end{equation}      
 We proceed in the proof by considering separately the cases $\alpha\geq 0$ and $\alpha<0$.
 \medskip
 
 {\bf The case $\alpha\geq 0$}.    Since $u$ is a sub solution and $v$ a super solution, by the positive 1--homogeneity of $F$ we have in this case
 $$
 \left\{
 \begin{array}{l}
-F\left(|\nabla_x \phi |^\alpha X_\epsilon \right) + |\nabla _x \phi |^{\beta} +\lambda |u|^\alpha u(x) \leq f(x)\\[1ex]
 -F\left(-|\nabla_y \phi |^\alpha Y_\epsilon \right) + |\nabla _y\phi |^{\beta} +\lambda |v|^\alpha v(y) \geq f(y)
 \end{array}\right.
 $$
 Subtracting the above inequalities and using also that $u(x)-v(y)>\phi (x,y)\geq 0$,  for any $t>0$ we can write
 $$
 \begin{array}{l}
 t\, F\left(|\nabla_x \phi |^\alpha X_\epsilon \right) -\left[ F\left((1+t)|\nabla_x \phi |^\alpha X_\epsilon \right) - F\left(-|\nabla_y \phi |^\alpha Y_\epsilon \right)\right]  \\[1ex]
\qquad \qquad \qquad \leq |\nabla _y\phi |^{\beta} - |\nabla _x\phi |^{\beta} +f(x)-f(y)\, ,
 \end{array}
 $$
and therefore
$$
 \begin{array}{ll}
 t\, |\nabla _x\phi |^{\beta} \leq  &   F\left((1+t)|\nabla_x \phi |^\alpha X_\epsilon \right) - F\left(-|\nabla_y \phi |^\alpha Y_\epsilon \right)\\[1ex] 
 & + |\nabla _y\phi |^{\beta} - |\nabla _x\phi |^{\beta} +t\, \left( f-\lambda |u|^\alpha u\right)^+ +f(x)-f(y)\, . 
 \end{array}
 $$
By the uniform ellipticity of $F$, it then follows that
 $$
 \begin{array}{ll}
 t\, |\nabla _x\phi |^{\beta} \leq  &   \mathcal{M}^+\left( (1+t)|\nabla_x \phi |^\alpha X_\epsilon +|\nabla_y \phi |^\alpha Y_\epsilon \right)\\[1ex] 
 & + |\nabla _y\phi |^{\beta} - |\nabla _x\phi |^{\beta} +t\, \left( f-\lambda |u|^\alpha u\right)^+ +f(x)-f(y)\, . 
 \end{array}
 $$
 
 By multiplying the right inequality of \eqref{eqine} on the left and on the right by
 $$
 \left(
 \begin{array}{cc}
 \sqrt{1+t}|\grad_x \phi|^{\alpha/2} I_N & O\\
 O & |\grad_y \phi|^{\alpha/2} I_N
 \end{array}
 \right)
 $$
 and testing the resulting inequality on vectors of the form $(v,v)$ with $v\in \R^N$, we further obtain
 $$
 (1+t) |\grad_x \phi|^\alpha X_\epsilon +|\grad_y \phi|^\alpha Y_\epsilon \leq Z_{\alpha,t} + O(\epsilon)\, ,
 $$
 with
 \begin{equation}\label{Zeta}
 Z_{\alpha,t} =  
 (1+t) |\grad_x \phi|^\alpha D^2_{xx} \phi 
 +  \sqrt{1+t} |\grad_x \phi|^{\alpha/2} |\grad_y \phi|^{\alpha/2}  \left[ D^2_{xy} \phi +\left( D^2_{xy}\phi\right)^t \right] +|\grad_y \phi|^\alpha D^2_{yy} \phi\, . 
\end{equation}
Hence,  after letting $\epsilon \to 0$, we get
 \begin{equation}\label{initial}
 t\, |\nabla _x\phi |^{\beta} 
 \leq   \mathcal{M}^+\left( Z_{\alpha,t} \right)  + |\nabla _y\phi |^{\beta} - |\nabla _x\phi |^{\beta} +t\, \left( f-\lambda |u|^\alpha u\right)^+ +f(x)-f(y)\, . 
  \end{equation}
 We now proceed by evaluating  the right hand side terms of \eqref{initial}.

An explicit computation shows that, setting $\eta = \frac{ |x-y|}{d(y)}$ and $\zeta =\frac{x-y}{|x-y|}\, ,$ one has

$$\grad_x \phi (x,y)=  \frac{k}{ d(y)^\tau}\left[ (L+ (1+ p) \xi^p) \zeta -p\,  \xi^{p+1} \nabla d (x)\right]\, , $$
 as well as
$$\grad_y \phi (x,y) = - \frac{k}{ d(y)^\tau} \left[ (L+ (1+ p) \xi^p) \zeta +\tau \, \eta \left( L + \xi^p\right) \nabla d(y)\right]\, .$$  
From now on we denote with $c$ possibly different positive constants which depend only on $p$, $N$, $a$, $A$,  $\alpha$ and $\beta$.
As discussed in \cite{CDLP}, for $L>1$ fixed suitably large  depending only on $p$, one has
 \begin{equation}\label{lbgradx}
 |\grad_x\phi|\geq c k \frac{ 1+\xi^{p+1}}{d(y)^\tau} 
  \end{equation}
  and 
 \begin{equation}\label{ubgrad}
 |\grad_x\phi|\,, \  |\grad_y\phi| \leq c k \frac{ 1+\xi^{p+1}}{d(y)^{\tau+1}}\, .
 \end{equation}
 Moreover, we notice that  one  has also
 \begin{equation}\label{lbgrady}
 |\grad_y\phi| \geq \frac{k}{d(y)^\tau} \left[ L+(1+p)\xi^p-\tau \eta (L+\xi^p)|\grad d(y)|\right]  \geq c k \frac{ 1+\xi^p}{d(y)^\tau}\qquad \hbox{ if \ } \tau\, \eta \leq \frac{1}{2}\, .
 \end{equation}
On the other hand, the second order derivatives of $\phi$ may be written as follows
 $$
 \begin{array}{c}
 D^2_{xx}\phi = \frac{k}{d(y)^\tau} \left\{ \frac{\left[ L+(1+p)\xi^p\right]}{|x-y|} B
  + p(1+p) \frac{\xi^{p-1}}{d(x)} T -p(1+p) \frac{ \xi^p}{d(x)} \left( C+C^t\right)\right. \\[2ex]
 \qquad  \qquad \left. +p(1+p) \frac{\xi^{p+1}}{d(x)} \grad d(x)\otimes \grad d(x) -p\,  \xi^{p+1}D^2 d(x)\right\}
 \end{array}
 $$
 $$
\begin{array}{c}
 D^2_{xy}\phi = - \frac{k}{d(y)^\tau} \left\{ \frac{\left[ L+(1+p)\xi^p\right]}{|x-y|} B
  + p(1+p) \frac{\xi^{p-1}}{d(x)} T -p(1+p) \frac{ \xi^p}{d(x)} C^t \right. \\[2ex]
 \qquad  \qquad \quad \left. + \frac{\tau\,  \left[ L+(1+p)\xi^p\right]}{d(y)} E - \frac{ \tau \, p\, \xi^{p+1}}{d(y)} \grad d(x)\otimes \grad d(y)\right\}
 \end{array}  $$
  $$
\begin{array}{c}
 D^2_{yy}\phi = \frac{k}{d(y)^\tau} \left\{ \frac{\left[ L+(1+p)\xi^p\right]}{|x-y|} B
  + p(1+p) \frac{\xi^{p-1}}{d(x)} T + \frac{\tau\, \left[ L+(1+p)\xi^p\right] }{d(y)} \left( E+E^t\right)\right. \\[2ex]
 \qquad  \qquad \left. + \frac{\tau\, (\tau+1)\, \eta (L+\xi^p)}{d(y)} \grad d(y)\otimes \grad d(y) - \tau\, \eta\, (L+\xi^p) D^2 d(y)\right\}
 \end{array}
$$
with $B=I_N -\zeta \otimes \zeta\,,\ T=\zeta \otimes \zeta\, ,\ C=\zeta \otimes \grad d(x)$ and $E=\zeta \otimes \grad d(y)\, .$ 

Therefore, the matrix $Z_{\alpha,t}$ defined in \eqref{Zeta} is given by
$$
\scriptstyle{\begin{array}{l}
Z_{\alpha,t} = \frac{k}{d(y)^\tau} \left\{ \left( \sqrt{1+t}|\grad_x \phi|^{\alpha/2}-|\grad_y \phi|^{\alpha/2}\right)^2 \left[ \frac{(L+(1+p)\xi^p)}{|x-y|} B
 +p(1+p)\frac{\xi^{p-1}}{d(x)}T\right]\right.\\[2ex]
\qquad  \qquad - \sqrt{1+t} |\grad_x \phi|^{\alpha/2} \left( \sqrt{1+t}|\grad_x \phi|^{\alpha/2}-|\grad_y \phi|^{\alpha/2}\right) \frac{p(1+p)\xi^p}{d(x)}\left( C+C^t\right)\\[2ex]
\qquad  \qquad  - |\grad_y\phi|^{\alpha/2} \left( \sqrt{1+t}|\grad_x \phi|^{\alpha/2}-|\grad_y \phi|^{\alpha/2}\right) \frac{\tau (L+(1+p)\xi^p)}{d(y)} \left( E+E^t\right)\\[2ex]
\qquad  \qquad +p(1+t)|\grad_x \phi|^\alpha \left[  \frac{(1+p)\xi^{p+1}}{d(x)} \grad d(x)\otimes \grad d(x)
-\xi^{p+1} D^2 d(x)\right]\\[2ex]
\qquad  \qquad + \sqrt{1+t}|\grad_x \phi|^{\alpha/2} |\grad_y \phi|^{\alpha/2}  \frac{\tau p \xi^{(p+1)}}{d(y)}\left[  \grad d(x)\otimes \grad d(y) +\grad d(y)\otimes \grad d(x) \right]\\[2ex]
\qquad  \qquad \left. + \tau\, |\grad_y \phi|^\alpha \left[  \frac{ (1+\tau )\eta (L+\xi^p)}{d(y)} \grad d(y)\otimes \grad d(y)- \eta (L+\xi^p) D^2 d(y)\right] \right\}\, ,
 \end{array}}
 $$
and, recalling that $\xi=|x-y|/d(x)$ and that $d<1$ in $B$, this yields the estimate
$$
\begin{array}{l}
\mathcal{M}^+ (Z_{\alpha,t})\leq \frac{c\, k}{d(y)^\tau} \left\{ \left( \sqrt{1+t}|\grad_x \phi|^{\alpha/2}-|\grad_y \phi|^{\alpha/2}\right)^2   \frac{1+\xi^p}{|x-y|}\right.\\[2ex]
\qquad \qquad + \sqrt{1+t}|\grad_x \phi|^{\alpha/2}\left|  \sqrt{1+t}|\grad_x \phi|^{\alpha/2}-|\grad_y\phi|^{\alpha/2}\right| \frac{\xi^p}{d(x)}\\[2ex]
\qquad \qquad + |\grad_y \phi|^{\alpha/2}\left|  \sqrt{1+t}|\grad_x \phi|^{\alpha/2}-|\grad_y\phi |^{\alpha/2}\right| \frac{1+\xi^p}{d(y)}\\[2ex]
\qquad \qquad  \left. +(1+t)|\grad_x \phi|^\alpha   \frac{\xi^{p+1}}{d(x)} + \sqrt{1+t}|\grad_x \phi|^{\alpha/2} |\grad_y \phi|^{\alpha/2}\frac{\xi^{p+1}}{d(y)} + |\grad_y \phi|^\alpha \eta \frac{1+\xi^p}{d(y)} \right\} .
\end{array}
$$
By observing that
$$
\left| \sqrt{1+t}|\grad_x \phi|^{\alpha/2}-|\grad_y \phi|^{\alpha/2}\right| \leq   (\sqrt{t+1}-1)  |\grad_x \phi|^{\alpha/2} +\left| |\grad_x \phi|^{\alpha/2}-|\grad_y \phi|^{\alpha/2}\right| 
$$
and  by   applying the trivial inequalities  $\sqrt{1+t}-1\leq  t$, $\sqrt{1+t} \left( \sqrt{t+1}-1\right)\leq t$, $\sqrt{1+t}\leq 1+t$, after rearranging terms we then deduce
\begin{equation}\label{ubmpiu1}
\begin{array}{l}
\mathcal{M}^+ (Z_{\alpha,t})\leq \frac{c\, k}{d(y)^\tau} \left\{  t^2 |\grad_x \phi|^\alpha  \frac{1+\xi^p}{|x-y|}\right. \\[2ex]
\qquad \qquad \qquad \quad + t\, \left[ |\grad_x \phi|^\alpha  \frac{\xi^p+\xi^{p+1}}{d(x)} + |\grad_x \phi|^{\alpha/2} |\grad_y \phi|^{\alpha/2} \left( \frac{1+\xi^p+\xi^{p+1}}{d(y)} +\frac{\xi^p}{d(x)}\right) \right]  \\[2ex]
\qquad \qquad \qquad \quad+ \left( |\grad_x \phi|^{\alpha/2}-|\grad_y \phi|^{\alpha/2}\right)^2  \frac{1+\xi^p}{|x-y|}
\\[2ex]
\qquad \qquad \qquad \quad+ \left| |\grad_x \phi|^{\alpha/2}-|\grad_y \phi|^{\alpha/2}\right|  \left(  |\grad_x \phi|^{\alpha/2}  \frac{\xi^p}{d(x)}
 + |\grad_y \phi|^{\alpha/2}  \frac{1+\xi^p}{d(y)}\right)   \\[2ex]
 \qquad \qquad  \qquad  \quad \left. + |\grad_x \phi|^\alpha   \frac{\xi^{p+1}}{d(x)} +|\grad_y \phi|^\alpha\frac{(1+\xi^p)\, \eta}{d(y)}  
+  |\grad_x \phi|^{\alpha/2}  |\grad_y \phi|^{\alpha/2}  \frac{\xi^{p+1}}{d(y)} \right\}
\end{array}
\end{equation}
We now recall that, as proved in \cite{CDLP}, for all $q, \gamma>0$ one has
\begin{equation}\label{coeff}
\frac{\xi^q}{d(x)^\gamma}\leq 2^\gamma \frac{1+\xi^{q+\gamma}}{d(y)^\gamma}\, .
\end{equation}
Moreover, if   $\alpha\geq 2$,   the mean value  theorem,   the bounds \eqref{ubgrad}, \eqref{coeff}  and the explicit 
expression of $\grad_x \phi +\grad_y \phi$ imply that
$$
\begin{array}{rl}
\left| |\grad_x \phi|^{\alpha/2}-|\grad_y \phi|^{\alpha/2}\right| 
\leq  & c\, \max \left\{ |\grad_x \phi|^{\alpha/2-1}, |\grad_y \phi|^{\alpha/2-1}\right\} \left| \grad_x \phi+\grad_y \phi\right|\\[2ex]
\leq & \displaystyle c k^{\alpha/2} \frac{(1+\xi^{p+1})^{\alpha/2-1}}{d(y)^{\alpha/2(\tau+1) -1}} \left( \frac{\xi^p}{d(x)}+\frac{1+\xi^p}{d(y)}\right) |x-y|\\[2ex]
\leq  & \displaystyle c \left[ k\frac{(1+\xi^{p+1})}{d(y)^{\tau+1}}\right]^{\alpha/2} |x-y|
\end{array}
$$   
Analogously, if $\alpha<2$ but $\tau \, \eta \leq 1/2$, from \eqref{lbgradx}, \eqref{lbgrady} and again \eqref{coeff} we deduce
$$
 \left| |\grad_x \phi|^{\alpha/2}-|\grad_y \phi|^{\alpha/2}\right| \leq  c   \left[ k\frac{(1+\xi^p)}{d(y)^\tau}\right]^{\alpha/2-1}   k \frac{(1+\xi^{p+1})}{d(y)^{\tau+1}}|x-y| 
 $$  
 and therefore, since $\xi\leq \frac{1}{2\tau -1}$ for $\eta\leq \frac{1}{2\tau}$, we obtain in this case
 $$
  \left| |\grad_x \phi|^{\alpha/2}-|\grad_y \phi|^{\alpha/2}\right| \leq  c  \left[ k\frac{(1+\xi^{p+1})}{d(y)^{\tau+2/\alpha}}\right]^{\alpha/2} |x-y|\, .
  $$
  Finally, if $\alpha<2$ and $\tau\, \eta >1/2$, that is $|x-y|>d(y)/2\tau$, we have
  $$
 \begin{array}{l}
  \displaystyle \left| |\grad_x \phi|^{\alpha/2}-|\grad_y \phi|^{\alpha/2}\right| \\[2ex]
  \displaystyle \qquad \leq      \left| \grad_x \phi+\grad_y \phi\right|^{\alpha/2} \leq
  c \left[ \frac{k(1+\xi^{p+1})|x-y|}{d(y)^{\tau+1}}\right]^{\alpha/2} \\[2ex]
\qquad  \displaystyle  \leq   c \left[ \frac{k(1+\xi^{p+1})}{d(y)^{\tau+1}}\right]^{\alpha/2}  \left(\frac{2\tau}{d(y)}\right)^{1-\alpha/2} |x-y| = c\left[ k\frac{(1+\xi^{p+1})}{d(y)^{\tau+2/\alpha}}\right]^{\alpha/2} |x-y|\, .
   \end{array}
  $$ 
  Thus, in all cases we obtain
  \begin{equation}\label{diffgrad}
   \left| |\grad_x \phi|^{\alpha/2}-|\grad_y \phi|^{\alpha/2}\right| \leq  c  \left[ k\frac{(1+\xi^{p+1})}{d(y)^{\tau+\max\{1,2/\alpha\}}}\right]^{\alpha/2} |x-y|\leq c \frac{|\grad_x \phi|^{\alpha/2}}{d(y)^{\max \{ \alpha/2,1\}}} |x-y|\, .
\end{equation}
By using inequalities \eqref{ubgrad}, \eqref{lbgradx}, \eqref{coeff} and \eqref{diffgrad}, from estimate \eqref{ubmpiu1} it then follows
\begin{equation}\label{ubmpiu2}
\mathcal{M}^+ (Z_{\alpha,t}) \leq \frac{c k |\grad_x\phi|^\alpha}{d(y)^\tau} \left\{ t^2 \frac{1+\xi^p}{|x-y|} +t \frac{1+\xi^{p+2}}{d(y)^{\alpha/2+1}}+|x-y| \frac{1+\xi^{p+2}}{d(y)^{\alpha+2}}\right\}\, .
\end{equation}
Moreover,  since $p\geq \frac{2+\alpha-\beta}{\beta-\alpha-1}$ and $\tau\geq \frac{\alpha+2}{2(\beta-\alpha-1)}$, by using again \eqref{lbgradx},
we further deduce
$$
\mathcal{M}^+ (Z_{\alpha,t}) \leq c k |\grad_x\phi|^\alpha  \left\{  t^2\frac{1+\xi^p}{d(y)^\tau |x-y|} +t \frac{ |\grad_x\phi|^{\beta-\alpha}}{k^{\beta-\alpha}}+
\frac{|x-y| (1+\xi^{p+2})}{d(y)^{\tau+ \alpha+2}}\right\}\, .
$$
Using the above inequality jointly with \eqref{initial} yields
$$
\begin{array}{rl}
 t\, |\nabla _x\phi |^{\beta-\alpha} 
 \leq   &  \displaystyle c k   \left\{  t^2\frac{1+\xi^p}{d(y)^\tau |x-y|} +t \frac{ |\grad_x\phi|^{\beta-\alpha}}{k^{\beta-\alpha}}+
\frac{|x-y| (1+\xi^{p+2})}{d(y)^{\tau+\alpha+2}}\right\}\\[2ex]
 & + |\nabla _x\phi |^{-\alpha}\left(|\nabla _y\phi |^{\beta} - |\nabla _x\phi |^{\beta}\right) +t\, |\nabla _x\phi |^{-\alpha}\left( f-\lambda |u|^\alpha u\right)^+ \\[2ex]
 & +|\nabla _x\phi |^{-\alpha}\left(f(x)-f(y)\right)\, ,
\end{array}
$$
and therefore, being $\beta>\alpha+1$, for $k$ sufficiently large one has
$$
\begin{array}{rl}
 \frac{t}{2} \, |\nabla _x\phi |^{\beta-\alpha} 
 -     t^2\frac{c k (1+\xi^p)}{d(y)^\tau |x-y|} \leq & \displaystyle  \frac{c k |x-y| (1+\xi^{p+2})}{d(y)^{\tau+\alpha+2}} 
 + |\nabla _x\phi |^{-\alpha}\left(|\nabla _y\phi |^{\beta} - |\nabla _x\phi |^{\beta}\right)\\[2ex]
 &  +t\, |\nabla _x\phi |^{-\alpha}\left( f-\lambda |u|^\alpha u\right)^+ + |\nabla _x\phi |^{-\alpha}\left(f(x)-f(y)\right)\, .
\end{array}
$$
We now choose $t>0$ in order to maximize the left hand side, namely 
$$t=\frac{|\grad_x\phi|^{\beta-\alpha} d(y)^\tau |x-y|}{4 c k (1+\xi^p)}\, .$$
We then obtain
$$
\begin{array}{rl}
 |\nabla _x\phi |^{2(\beta-\alpha)} 
   \leq & \!\!\!\!\! \displaystyle c\left\{  \frac{  k ^2  (1+\xi^{2(p+1)})}{d(y)^{2\tau+\alpha+2}} 
 + k |\nabla _x\phi |^{-\alpha}\frac{(1+\xi^p)\left(|\nabla _y\phi |^{\beta} - |\nabla _x\phi |^{\beta}\right)}{|x-y|d(y)^\tau} \right.\\[2ex]
 &  \displaystyle \left.+  |\nabla _x\phi |^{\beta-2\alpha}\left( f-\lambda |u|^\alpha u\right)^+ + k |\nabla _x\phi |^{-\alpha}\frac{(1+\xi^p) \left(f(x)-f(y)\right)}{|x-y| d(y)^\tau}\right\}\, .
\end{array}
$$
Moreover, arguing as for \eqref{diffgrad} in the case $\alpha\geq 2$, we also have
$$
\left| |\nabla _y\phi |^{\beta} - |\nabla _x\phi |^{\beta} \right| \leq c k^\beta |x-y|\frac{(1+\xi^{p+1})^\beta}{d(y)^{(\tau+1)\beta}}\, ,
$$
so that
$$
\begin{array}{rl}
 |\nabla _x\phi |^{2(\beta-\alpha)} 
   \leq & \!\!\!\!\! \displaystyle C\left\{  \frac{k ^2  (1+\xi^{(p+1)})^2}{d(y)^{2\tau+\alpha+2}} 
 +  |\nabla _x\phi |^{-\alpha}\frac{k^{\beta+1}(1+\xi^{p+1})^{\beta+1}}{d(y)^{\tau(\beta+1)+\beta}}\right.\\[2ex]
 &  \displaystyle \left.+  |\nabla _x\phi |^{\beta-2\alpha}  +  |\nabla _x\phi |^{-\alpha}\frac{k (1+\xi^p)}{ d(y)^\tau}\right\}\, ,
\end{array}
$$
for some constant $C>0$ depending now also on $\| (f-\lambda|u|^\alpha u)^+\|_\infty$ and on the Lipschitz constant of $f$.

By inequality \eqref{lbgradx}  it then follows
$$
\begin{array}{rl}
|\nabla _x\phi |^{2(\beta-\alpha)} & \displaystyle \leq C \left\{ \frac{|\nabla _x\phi |^2}{d(y)^{\alpha+2}}+ \frac{|\nabla _x\phi |^{\beta-\alpha+1}}{d(y)^\beta} + |\nabla _x\phi |^{\beta-2\alpha} +|\nabla _x\phi |^{1-\alpha}\right\}\\[2ex]
&\displaystyle  \leq C \left\{ \frac{ |\nabla _x\phi |^{2+\frac{\alpha+2}{\tau}}}{k^{\frac{\alpha+2}{\tau}}} +\frac{ |\nabla _x\phi |^{ \beta-\alpha+1+\frac{\beta}{\tau}}}{k^{\frac{\beta}{\tau}}} +  |\nabla _x\phi |^{\beta-2\alpha} \right\} \, ,
\end{array}
$$
Recalling that $\alpha>0$, $\beta>\alpha +1$ and $\tau=\frac{\beta}{\beta-\alpha-1}$, the last inequality  gives  a contradiction for $k$ large enough.
\medskip

{\bf The case $\alpha< 0$.} As proved in \cite{BDcocv}, if $\alpha<0$ a sub solution $u$ and super solution $v$ of equation \eqref{eqq1}  satisfy respectively
in the viscosity sense
$$
\left\{
\begin{array}{c}
\displaystyle -F(D^2u) +|\grad u|^{\beta-\alpha} +\lambda |\grad u|^{-\alpha} |u|^\alpha u\leq |\grad u|^{-\alpha} f\\[2ex]
\displaystyle -F(D^2v) +|\grad v|^{\beta-\alpha} +\lambda |\grad v|^{-\alpha} |v|^\alpha v\geq |\grad v|^{-\alpha} f
\end{array}
\right.
$$
From \eqref{eqine} in this case  it then follows that
$$
\left\{
\begin{array}{c}
\displaystyle -F(X_\epsilon) +|\grad_x \phi|^{\beta-\alpha} +\lambda |\grad_x \phi|^{-\alpha} |u|^\alpha u(x)\leq |\grad_x \phi|^{-\alpha} f(x)\\[2ex]
\displaystyle -F(-Y_\epsilon) +|\grad_y \phi|^{\beta-\alpha} +\lambda |\grad_y\phi|^{-\alpha} |v|^\alpha v(y)\geq |\grad_y\phi|^{-\alpha} f(y)
\end{array}
\right.
$$
and, arguing as in the previous case, we now obtain for any $t>0$
$$
\begin{array}{ll}
 t\, |\nabla _x\phi |^{\beta-\alpha} \leq  &   \mathcal{M}^+\left( Z_{0,t} \right) + |\nabla _y\phi |^{\beta-\alpha} - |\nabla _x\phi |^{\beta-\alpha}+t\, |\grad_x \phi|^{-\alpha} \left( f-\lambda |u|^\alpha u\right)  \\[2ex] 
 & -  \left( f-\lambda |u|^\alpha u\right) \, \left( |\grad_y\phi|^{-\alpha}- |\grad_x\phi|^{-\alpha}\right) 
 + |\grad_y \phi|^{-\alpha} \left( f(x)-f(y)\right) \, ,\\[2ex]
 \end{array}
$$
where $Z_{0,t}$ is defined by \eqref{Zeta} (with $\alpha=0$).

By applying inequalities \eqref{ubmpiu2} (with $\alpha=0$), \eqref{lbgradx}, \eqref{ubgrad}, \eqref{diffgrad}, in the present case,  taking into account that $\beta-\alpha>1$ and $0<-\alpha<1$, we deduce that
  $$
 |\nabla _x\phi |^{2(\beta-\alpha)} \leq C\, \left\{ \frac{|\grad_x\phi|^2}{d(y)^2} + \frac{|\grad_x\phi|^{\beta-\alpha+1}}{d(y)^{\beta-\alpha}} 
+|\grad_x\phi|^{\beta-2\alpha} \right\}\, ,
$$
for some constant $C>0$ depending on $p, \alpha, \beta, a, A, N, \|f-\lambda |u|^\alpha u\|_\infty$  and on the Lipschitz constant of $f$. Since now $\tau=\frac{\beta-\alpha}{\beta-\alpha-1}$,  we reach a contradiction for $k$ sufficiently large as before.
 \end{proof}

As  in \cite{CDLP}, the  above Proposition and a scaling argument for solutions of equation \eqref{eqq1}  give  the following result.

  \begin{theo}\label{lipsol}  Let $F$  satisfy \eqref{unifel} and, for $\lambda\geq 0, \alpha>-1$ and $\beta>\alpha+1$, let $u$  be a continuous solution of equation \eqref{eqq1} in $\Omega\subset \R^N$, with $f$ Lipschitz continuous in $\Omega$. Then, $u$ is locally Lipschitz continuous in $\Omega$ and  there exists a positive constant $M$, depending only on $\alpha, \beta, a, A, N, \|f-\lambda |u|^\alpha u\|_\infty$ and on the Lipschitz constant of $f$, such that at any differentiability point $x\in \Omega$ one has
  $$
 |\grad u(x)|  \leq  \frac{M}{{\rm dist}_{\partial \Omega}(x)^{\frac{1}{\beta-\alpha-1}}}\, .
  $$
 \end{theo}
\section{Proof of Theorem \ref{sympaetpas}.}\label{theo1}

By using the Lipschitz estimates obtained in the previous section, we can now prove Theorem \ref{sympaetpas}.

\begin{proof}[Proof of Theorem \ref{sympaetpas}] 
Let $u_\lambda$ be a solution of \eqref{eq45}. We begin by giving a bound that will be useful in the whole proof.
Observe that $u_\lambda^+$ is a sub solution of
$$-| \nabla u_\lambda^+ |^\alpha F( D^2 u_\lambda^+) \leq |f|_\infty;$$
from known estimates, see \cite{BD1}, this  implies that 
\begin{equation}\label{upper}|u_\lambda ^+|_\infty \leq c_1 |f|_\infty ^{\frac{1}{1+ \alpha}}.\end{equation}
Let us consider first the case  when there exists a sub solution  $\varphi$  of \eqref{eq3}. 
Then,  $\varphi-| \varphi |_\infty$ is  a sub solution of  equation \eqref{eq45}, and by the 
comparison principle   we deduce
 $u_\lambda \geq \varphi-| \varphi|_\infty$. Thus, 
in this case $(u_\lambda)$ is uniformly bounded in $\Omega$. The Lipschitz estimates in Theorem \ref{existence}  then yield that $u_\lambda$ is uniformly converging up to a sequence to a Lipschitz solution of  problem (\ref{eq3}). 

\medskip
We now treat the second case, i.e. we suppose that  \eqref{eq3} has no solutions.
In particular $|u_\lambda|_\infty $ diverges, since otherwise we could extract from $(u_\lambda)$ a subsequence 
 converging to a solution of \eqref{eq3}.  

On the other hand,  since  $-\left(\frac{ |f|_\infty }{\lambda}\right) ^{\frac{1}{1+ \alpha}}$ is a sub solution of \eqref{eq45}, by the comparison principle we obtain 
$u_\lambda^- \leq \left(\frac{|f|_\infty}{\lambda}\right) ^{\frac{1}{1+ \alpha}}$, which,    
jointly with  \eqref{upper}, yields $\lambda |u_\lambda|_\infty^{1+ \alpha} \leq  c_1|f|_\infty.$ 
Hence,   there exists  $(x_\lambda)\subset \Omega$ such that $u_\lambda (x_\lambda)=-|u_\lambda|_\infty\rightarrow -\infty$ and there exists a constant $c_\Omega\geq 0$ such that, up to a subsequence, $\lambda |u_\lambda|_\infty^{1+ \alpha}\rightarrow c_\Omega$.
   
We will show, as in \cite{LL} (see also \cite{LP} and \cite{P}), that  $v_\lambda = u_\lambda + | u_\lambda|_\infty=u_\lambda -u_\lambda(x_\lambda)$ converges up to a subsequence to  a   function $v$ such that the pair $(c_\Omega,v)$ solves \eqref{ergodic}.

Clearly,  $v_\lambda$ satisfies  in $\Omega$
$$
-|\nabla v_\lambda |^\alpha F( D^2 v_\lambda) + | \nabla v_\lambda|^\beta + \lambda (v_\lambda)^{1+ \alpha } = f+ \lambda ( v_\lambda^{\alpha+1}-|u_\lambda|^\alpha u_\lambda) \geq f \, .
$$
Next, we set
$$\gamma=\frac{2+\alpha-\beta}{\beta-1-\alpha}\, ,$$
 and for  $s, \delta_o>0$ to be chosen sufficiently small, let us consider the function
\begin{equation}\label{phi}
\begin{array}{l}
\displaystyle \phi(x)=\frac{\sigma}{(d(x)+s)^\gamma}- \frac{\sigma}{(\delta_o+s)^\gamma} \quad \hbox{ if } \gamma>0\, ,\\[2ex]
\phi(x)=-\sigma\, \log (d(x)+s) + \sigma\, \log (\delta_o+s) \quad \hbox{ if } \gamma=0\, ,
\end{array}
\end{equation}
where  $\sigma=  \left((\gamma+1)\frac{a}{2} \right)^{\frac{1}{\beta-\alpha-1}}\gamma^{-1}$ if $\gamma>0$,  
$\sigma= \frac{a}{2}$ if $\gamma=0$.
A direct computation shows that, for $d(x)\leq \delta_0$ with $\delta_o$ small enough, in the case $\gamma>0$ one has
$$ -|\nabla \phi|^\alpha \mathcal{M}^-( D^2 \phi) + | \nabla \phi|^\beta + \lambda \phi^{1+ \alpha } \leq 
- \frac{a(\sigma \gamma)^{\alpha+1}}{2(d+s)^{(\gamma+1)\beta}}+\frac{A (\sigma \gamma)^{\alpha+1}|D^2d|_\infty}{(d+s)^{(\gamma+1)(\alpha+1)}}
+ \frac{\lambda \sigma^{\alpha+1}}{(d+s)^{\gamma(\alpha+1)}},$$
 and,  in the case $\gamma=0$,
$$ -|\nabla \phi|^\alpha \mathcal{M}^-( D^2 \phi) + | \nabla \phi|^\beta + \lambda \phi^{1+ \alpha } \leq 
- \frac{\sigma ^{\alpha+2}}{(d+s)^{\alpha+2}}+\frac{A \sigma^{\alpha+1}|D^2d|_\infty}{(d+s)^{(\alpha+1)}}
+ \lambda \left( -\sigma \log (d+s)\right)^{\alpha+1}.$$
In both cases,  by the ellipticity of $F$ and for $\delta_o$ and $s$ sufficiently small, we obtain
$$ -|\nabla \phi|^\alpha F( D^2 \phi) + | \nabla \phi|^\beta + \lambda \phi^{1+ \alpha } \leq -|f|_\infty\leq f(x) \quad \hbox{ in } \Omega\setminus \overline{\Omega}_{\delta_0}\, .$$
Moreover, one has $\phi=0\leq v_\lambda$ on $\partial \Omega_{\delta_o}$ and $\phi\leq |u_\lambda|_\infty=v_\lambda$ on $\partial \Omega$ for $\lambda$ sufficiently small in dependence of $s$. The comparison principle then yields
\begin{equation}\label{lowbo}
v_\lambda \geq \phi >0 \quad \hbox{ in } \Omega\setminus \overline{\Omega}_{\delta_0}\, .
\end{equation}
Since $v_\lambda(x_\lambda)=0$, from \eqref{lowbo} we deduce that $(x_\lambda)\subset \overline{\Omega}_{\delta_o}$. The interior Lipschitz estimate of Theorem \ref{lipsol} then yields that $v_\lambda =u_\lambda -u_\lambda(x_\lambda)$ is  locally uniformly bounded and locally uniformly Lipschitz continuous.
This proves both statement \ref{a} of the theorem  and   that $(v_\lambda)$ is locally uniformly converging up to a subsequence to a Lipschitz continuous function $v_o\geq 0$ in $\Omega$. Moreover, since also $(x_\lambda)$ converges up to a subsequence to some point  $x_o\in \overline{\Omega}_{\delta_0}$, we obtain  $v_o(x_o)=0$. We observe further that, locally uniformly in $\Omega$, one has
$$
\lim_{\lambda \to 0} \lambda |u_\lambda|^\alpha u_\lambda  =\lim_{\lambda \to 0} \lambda \, |u_\lambda|_\infty^{\alpha+1}\,  
\frac{ |v_\lambda -|u_\lambda|_\infty|^{\alpha}(v_\lambda-|u_\lambda|_\infty)}{|u_\lambda|_\infty^{\alpha+1}}= -c_\Omega\, .
$$
This yields statement \ref{b}  and, letting $\lambda\to 0$ in the equation   satisfied by $v_\lambda$, also that $v_o$ is a viscosity solution of
$$
-|\nabla v_o |^\alpha F( D^2 v_o) + | \nabla v_o|^\beta = f+c_\Omega\, .
$$
Finally, letting $\lambda\to 0$ in inequality \eqref{lowbo}, we obtain $v_o \geq \phi$   in  $\Omega\setminus \overline{\Omega}_{\delta_0}$, which 
 in turn implies, by letting $s\to 0$ and $x\to \partial \Omega$, that $v_o(x)\to +\infty$ as $d(x)\to 0$. This completely proves statement \ref{c} and concludes the proof of the theorem.
    
     \end{proof}

\section{Existence of explosive solutions.}\label{exiergolambda}
  In this section we prove the existence of solutions of \eqref{eqq1} blowing up at the boundary. 
The results obtained here will be used in the proof of the existence of  ergodic pairs.

Throughout  the present section, on the one hand  we assume the regularity condition \eqref{smooth}, but,
 on the other  hand, we drop the assumption on the boundedness of the right hand side $f$, and we  consider continuous functions in $\Omega$, possibly unbounded as $d(x)\to 0$. 
 
For what follows, we set
$$\gamma = \frac{ 2+ \alpha-\beta }{ \beta-1-\alpha }\, ,$$
and  we denote by $C(x)$ a non negative function of class $\mathcal{C}^2$ in $\Omega$ satisfying  in a neighborhood of $\partial \Omega$
\begin{equation}\label{C(x)}
\begin{array}{lcl}C(x)=\left((\gamma+1)F(\grad d(x)\otimes\grad d(x))\right)^{\frac{1}{ \beta-\alpha-1}}\gamma^{-1}\  &\mbox{if }& \gamma>0,\\[1ex]
 C(x)=F(\grad d(x)\otimes\grad d(x))\quad & \mbox{if }& \gamma=0.
 \end{array}
\end{equation}
\begin{theo}\label{exilambdainf} Let  $\beta  \in (\alpha+1,\alpha+2]$,  $\lambda >0$ and  let $F$ satisfy \eqref{unifel} and \eqref{smooth}. Let further $f\in \mathcal{C}(\Omega)$  be  bounded from below  and such  that  
\begin{equation}\label{f}
\lim_{d(x)\to 0} f(x) d(x)^{\frac{\beta}{\beta-1-\alpha }-\gamma_0} = 0\, ,
\end{equation} 
for some $ \gamma_0\geq 0$.
 Then, there exists 
 a  solution  $u$   of
                     \begin{equation}\label{eqlambdainfini}\left\{ \begin{array}{cl}
                     -| \nabla u |^\alpha F(D^2 u) + |\nabla u |^\beta   + \lambda |u|^\alpha u  = f &  \hbox{ in } \ \Omega \, ,\\
                     u=+\infty  & \hbox{ on } \ \partial \Omega\, .
                     \end{array}\right.\end{equation}
Moreover, any solution $u$ of \eqref{eqlambdainfini} satisfies:  for any $\nu >0$ and for any $0\leq \gamma_1\leq \gamma_0$, with $\gamma_1< 1$, and $\gamma_1<\gamma$ when $\gamma>0$, there exists $D=\frac{D_1}{\lambda^{1/(\alpha+1)}}$, with $D_1>0$ depending on $\nu, \gamma_1, \alpha, \beta, a, A, |d|_{\mathcal{C}^2(\Omega)}, |C|_{\mathcal{C}^2(\Omega)}$ and on $f$, such that, for all $x\in \Omega$,
\begin{equation}\label{bouest}
\begin{array}{c}
\displaystyle \frac{C(x)}{d(x)^\gamma}-\frac{\nu}{d(x)^{\gamma-\gamma_1}}-D\leq u(x)\leq \frac{C(x)}{d(x)^\gamma}+\frac{\nu}{d(x)^{\gamma-\gamma_1}}+D
\ \hbox{ if } \gamma>0\, ,\\[2ex]
|\log d(x)|\left( C(x) - \nu d(x)^{\gamma_1}\right)-D\leq u(x)\leq  |\log d(x)|\left( C(x) +\nu d(x)^{\gamma_1}\right)+D
\ \hbox{ if } \gamma=0\, .
\end{array}
\end{equation}
 Furthermore, the solution $u$ is unique
\begin{itemize}
\item for $\alpha \geq 0$ and any $\beta $,  
\item for  $\alpha <0$  and any  $\beta >{\frac{1-\alpha-\alpha^2}{1-\alpha}}$, provided that
 $f$ satisfies \eqref{f} with $\gamma_0>-\alpha\, \gamma$.
\end{itemize}
 \end{theo}
 \begin{proof}  We  give the proof in the case $\gamma>0$, the reader can easily see the changes to be made 
when $\gamma=0$. 
\smallskip

{\bf 1. Estimates and  existence.}

 We will get the conclusion by showing that, for every  $\nu>0$ and for any $0\leq \gamma_1\leq \gamma_0$, with $\gamma_1< \min\{1, \gamma\}$, there exist  $D= \frac{D_1}{\lambda^{1/\alpha+1}}>0$, a super solution $\overline{w}$ and a sub solution $\underline{w}^s$ satisfying
$$
C(x) (d+ s) ^{-\gamma} -\nu (d+ s) ^{-\gamma+ \gamma_1}- D \leq \underline{w}^s(x) \leq  \overline{ w}(x)  
 \leq  C(x) d^{-\gamma} +\nu d^{-\gamma+ \gamma_1}+ D, 
 $$
  for  any $s>0$ sufficiently small.
   
  Assume  for a while that  this is proved. Then,  problem \eqref{eqlambdainfini} does have  solutions and any solution $u$ of \eqref{eqlambdainfini}  satisfies 
    $ \underline{w}^o \leq u \leq  \overline{w}$. 
   Indeed, for any $R>0$, we can consider  the solution $u_R$ of 
$$\left\{ \begin{array}{lc}
  -| \nabla u_R |^\alpha F( D^2 u_R) + | \nabla u_R |^\beta + \lambda |u_R|^\alpha u_R = f_R & {\rm in} \ \Omega\\
   u_R=R & {\rm on} \ \partial \Omega\, ,
   \end{array}
   \right.$$
   with $f_R=\min \{ f, R\}$.
 By Theorem \ref{comp}, $u_R$ is  monotone increasing with respect to $R$ and satisfies $\underline{w}^s\leq u_R \leq \bar w$, provided that $R> \max_{\partial \Omega} \underline{w}^s(x)$. Moreover, $u_R$ is locally uniformly Lipschitz continuous by Theorem \ref{existence}.
Thus, $u_R$ is  locally uniformly convergent as $R\to +\infty$ to a solution $\underline{u}$ of \eqref{eqlambdainfini}
 such that  $\underline{w}^0\leq \underline{ u}\leq \bar w$. 
  By definition, $\underline{u}$ is the so called minimal explosive solution.    The maximal explosive solution $\overline u$ is then obtained  as the limit for $\delta\rightarrow 0$ of the minimal explosive solutions in $\Omega_\delta$. 
Thus, it follows that for any solution $u$ of problem \eqref{eqlambdainfini} one has
 $$
     C(x) d^{-\gamma}-\nu d^{ -\gamma+ \gamma_1} -D \leq \underline{u}\leq u\leq \overline{u} \leq C(x) d^{-\gamma}+\nu d^{ -\gamma+ \gamma_1} +D\, .
$$
Let us now  proceed to the construction of $\overline{w}$ and $\underline{w}^s$. 

 Let $\delta>0$ be so small  that  in  the set $\Omega\setminus\Omega_{2\delta}=\{d (x)< 2\delta\}$  the function $d$ satisfies   $|\nabla d|=1$ and $C$ satisfies \eqref{C(x)}.  For $x\in  \Omega\setminus\Omega_{2\delta}$, let us consider the function 
 $$\varphi (x)= C(x)  d(x)^{-\gamma} +  \nu d(x)^{\gamma_1-\gamma}\, .$$
One has
 $$\nabla \varphi (x) = -\gamma\, C(x)  d^{-\gamma-1} \left(1+ \nu \frac{( \gamma-\gamma_1)}{ \gamma\, C(x)} d^{ \gamma_1} \right) \nabla d+ d^{-\gamma} \nabla C$$
 and
 $$  \begin{array}{rl}
   D^2 \varphi (x)= &  \gamma\,  ( \gamma+1) C(x) d^{-\gamma-2} \left( 1+ \nu \frac{( \gamma-\gamma_1) ( \gamma-\gamma_1+1)}{ \gamma \,( \gamma+1) C(x)} d^{ \gamma_1}\right ) \nabla d \otimes \nabla d\\[2ex]
 &  -\gamma\,  C(x)  d^{-\gamma-1} \left( 1+ \nu \frac{(\gamma-\gamma_1)}{ \gamma\, C(x)} d^{ \gamma_1} \right) D^2 d \\[2ex]
&  -\gamma \, d^{-\gamma-1} \left( \nabla d \otimes \nabla C+ \nabla C \otimes \nabla d \right) + d^{-\gamma}D^2C\, .
   \end{array}
   $$
By ellipticity of $F$ and by definition of $C(x)$ it then follows
$$
\begin{array}{rl}
F(D^2\varphi) \leq &  \frac{\gamma\,  ( \gamma+1) C(x)}{ d^{\gamma+2}} \left( 1+ \nu \frac{( \gamma-\gamma_1) ( \gamma-\gamma_1+1)}{ \gamma \,( \gamma+1) C(x)} d^{ \gamma_1}\right ) F\left( \nabla d \otimes \nabla d\right) +\frac{K_1}{ d^{\gamma+1}}\\[2ex]
= &  \frac{(\gamma\,   C(x))^{\beta-\alpha}}{ d^{\gamma+2}} \left( 1+ \nu \frac{( \gamma-\gamma_1) ( \gamma-\gamma_1+1)}{ \gamma \,( \gamma+1) C(x)} d^{ \gamma_1}\right ) +\frac{K_1}{ d^{\gamma+1}}\, ,
\end{array}
$$
for a constant $K_1>0$ depending on $\nu, \alpha, \beta, \gamma_1, a, A, |D^2d|_\infty$  
 and 
 $|C|_{\mathcal{C}^2(\Omega)}$. In what follows we denote by $K_i$, $i=1,2\ldots$, different constants depending on these quantities.
 
 Hence, we obtain
$$ \begin{array}{l}
- | \nabla \varphi |^{ \alpha} F ( D^2 \varphi)+ | \nabla \varphi |^\beta\\[2ex]
 \geq  |\nabla \varphi|^{\alpha} \left[ - \frac{(\gamma\,   C(x))^{\beta-\alpha}}{ d^{\gamma+2}} \left( 1+ \nu \frac{( \gamma-\gamma_1) ( \gamma-\gamma_1+1)}{ \gamma \,( \gamma+1) C(x)} d^{ \gamma_1}\right ) -\frac{K_1}{ d^{\gamma+1}} +|\nabla \varphi|^{\beta-\alpha}\right]\\[2ex]
\geq  |\nabla \varphi|^{\alpha} \left[ - \frac{(\gamma\,   C(x))^{\beta-\alpha}}{ d^{\gamma+2}} \left( 1+ \nu \frac{( \gamma-\gamma_1) ( \gamma-\gamma_1+1)}{ \gamma \,( \gamma+1) C(x)} d^{ \gamma_1}\right ) -\frac{K_1}{ d^{\gamma+1}} \right.\\[2ex]
\qquad \qquad \left. + \frac{(\gamma\, C(x))^{\beta-\alpha}}{d^{(\gamma+1)(\beta-\alpha)}} \left( 1+ \nu  \frac{( \gamma-\gamma_1)}{ \gamma\, C(x)} d^{ \gamma_1} \right)^{\beta-\alpha} -\frac{K_2}{d^{(\gamma+1) (\beta-\alpha)-1}}\right]\\[2ex]
\geq   |\nabla \varphi|^{\alpha} \left[ 
- \frac{(\gamma\,   C(x))^{\beta-\alpha}}{ d^{\gamma+2}} \left( 1+ \nu \frac{( \gamma-\gamma_1) ( \gamma-\gamma_1+1)}{ \gamma \,( \gamma+1) C(x)} d^{ \gamma_1}\right )
-\frac{K_1}{ d^{\gamma+1}}\right.\\[2ex]
\qquad \qquad \left.   + \frac{(\gamma\, C(x))^{\beta-\alpha}}{d^{(\gamma+1)(\beta-\alpha)}} \left( 1+ \nu  \frac{(\beta-\alpha)( \gamma-\gamma_1)}{ \gamma\, C(x)} d^{ \gamma_1} \right) -\frac{K_2}{d^{(\gamma+1) (\beta-\alpha)-1}}\right]
\end{array}
 $$
Recalling that $\gamma=\frac{\alpha+2-\beta}{\beta-\alpha-1}$, we finally deduce
$$
\begin{array}{l}
- | \nabla \varphi |^{ \alpha} F ( D^2 \varphi)+ | \nabla \varphi |^\beta\\[2ex]
 \geq  |\nabla \varphi|^{\alpha} \left[ 
 \nu  \frac{( \gamma-\gamma_1) (1+\gamma_1)}{ ( \gamma+1)}\frac{(\gamma\,   C(x))^{\beta-\alpha-1}}{ d^{\gamma+2-\gamma_1}} 
-\frac{K_3}{ d^{\gamma+1}}\right]\\[2ex]
=\frac{|(\gamma\, C(x)+\nu (\gamma-\gamma_1)d^{\gamma_1})\nabla d-d\, \nabla C|^\alpha}{d^{\frac{\beta}{\beta-\alpha-1}-\gamma_1}} \left[  \nu  \frac{( \gamma-\gamma_1) (1+\gamma_1)}{ ( \gamma+1)}(\gamma\,   C(x))^{\beta-\alpha-1}  
-K_3 d^{1-\gamma_1}\right]
\end{array}
$$
Since $\gamma_1\leq \gamma_0$, by assumption \eqref{f} on $f$   the last inequality implies that, for $\delta$ sufficiently small,
$$
- | \nabla \varphi |^{ \alpha} F ( D^2 \varphi)+ | \nabla \varphi |^\beta\geq f(x)\quad \hbox{ in } \Omega\setminus \Omega_{2\delta}\, ,
$$
and therefore also that
$$
- | \nabla \varphi |^{ \alpha} F ( D^2 \varphi)+ | \nabla \varphi |^\beta +\lambda \varphi^{\alpha+1} \geq f(x)\quad \hbox{ in } \Omega\setminus \Omega_{2\delta}\, .
$$
Clearly, the same inequality holds also for $\varphi_1(x)=\varphi(x)+D$, for any $D>0$. 

Next, for $x\in \{ \delta\leq d(x)\leq 2\delta\}$, we consider the function
$$
\varphi_2(x)= \delta^{-\gamma}C(x) e^{\frac{1}{d(x)-2\delta}+\frac{1}{\delta}}+E\, ,
$$
with $E>0$ to be conveniently fixed. A direct computation shows that, for $\delta$ sufficiently small in dependence of $|\nabla C|_\infty$, one has
$$| \nabla \varphi_2  |^\alpha |F( D^2  \varphi_2)| + | \nabla \varphi_2 |^\beta \leq K_4\quad \hbox{ in } \overline \Omega_{\delta}\setminus \Omega_{2\delta}\, .$$
Thus, if $E$ is chosen satisfying $E\geq \left( \frac{|f|_{L^\infty (\Omega_\delta)}+K_4}{\lambda}\right)^{\frac{1}{\alpha+1}}$, then we obtain
$$
- | \nabla \varphi_2 |^{ \alpha} F ( D^2 \varphi_2)+ | \nabla \varphi_2 |^\beta +\lambda \varphi_2^{\alpha+1} \geq f(x)\quad \hbox{ in } \overline\Omega_\delta \setminus \Omega_{2\delta}\, .
$$
Let us further restrict the smallness of $\delta$ by requiring that
\begin{equation}\label{delta1} 
 A\gamma \delta^{-\gamma-1}+ (\gamma-\gamma_1) \delta^{ \gamma_1-\gamma-1} < a \delta^{-2-\gamma}\, .
  \end{equation} 
 Then, setting $D=\left( \frac{|f|_{L^\infty (\Omega_\delta)}+K_4}{\lambda}\right)^{\frac{1}{\alpha+1}}$ and $E=\nu \delta^{\gamma_1-\gamma}+D$, the function
$$ \overline w (x)=  \left\{ \begin{array}{ll}
            C(x) d(x)^{-\gamma}  + \nu d^{-\gamma+ \gamma_1} + D & {\rm for } \  d <  \delta\\[1ex]
         \delta^{-\gamma}  C(x) e^{\frac{1}{d(x)-2\delta} +\frac{1}{\delta}} + E    & { \rm for }\  \delta\leq  d\leq 2\delta\\[1ex]
             E &{\rm for } \ d> 2\delta
             \end{array}\right.$$
 is the required  super solution in $\Omega$. Indeed,  in the set $\Omega_\delta$,  
 $\overline w$ is of class ${ \cal C}^2$ and it is a super solution by the properties of $\varphi_2$ and by the fact that locally constant functions satisfy
 $|\nabla u |^\alpha F( D^2 u) = 0$. On the other hand,  $\overline w$ is a super solution in $\Omega\setminus \Omega_{2 \delta}$ by the properties of $\varphi_1$ and $\varphi_2$ and the fact that $\varphi_2 (x)<\varphi_1 (x)$ for $d(x)>\delta$, as it follows  by evaluating  at points $x$ such that $d(x)=\delta$ the derivatives of $\varphi_1$ and $\varphi_2$ along the direction $\nabla d(x)$  and by using  \eqref{delta1}.                                                 
\medskip
           
As far as the sub solution is concerned,  for $s>0$, $\nu >0$ and $x\in \Omega\setminus \Omega_{2\delta}$, let us consider the function
 $$\varphi ^s (x)= C(x) (d(x)+ s)^{-\gamma}- \nu (d(x)+s) ^{-\gamma+ \gamma_1}\, .$$
Analogous computations as above give that
$$
\begin{array}{l}
- |\nabla \varphi^s|^{ \alpha} F ( D^2 \varphi^s)+ |\nabla \varphi^s|^\beta +\lambda |\varphi^s|^\alpha \varphi^s\\[2ex]
\leq |\nabla \varphi^s|^\alpha 
\left[ -\frac{(\gamma\,   C(x))^{\beta-\alpha-1}}{ (d+s)^{\gamma+2-\gamma_1}} \nu \frac{(\gamma-\gamma_1)(1+\gamma_1)}{\gamma+1}
+\frac{K_5}{ (d+s)^{\gamma+1}}\right] + \lambda \frac{A^{\alpha+1}}{(d+s)^{\gamma\, (\alpha+1)}}\leq f(x)\, ,
\end{array}
$$
for $\delta$ and $s$ sufficiently small, since $f$ is bounded from below.

Moreover, setting $C_\delta= \max_{\{ d(x)= \delta\}} C(x)$ and choosing $D\geq C_\delta \delta^{-\gamma}+\left(\frac{|f^-|_\infty}{\lambda}\right)^{\frac{1}{\alpha+1}}$,  the constant function $C_\delta(\delta+s)^{-\gamma}-\nu (\delta+s)^{-\gamma+\gamma_1}-D$ is also a sub solution in $\Omega$.

Therefore, the function
             $$ \underline{w}^s(x) = \left\{ \begin{array} {ll}
    \begin{array}{c} \max \left\{ C(x) (d+s)^{-\gamma}- \nu (d+s)^{ -\gamma+\gamma_1}-D\right. ,\\[1ex]
    \left. \ \ \ C_\delta( \delta+s)^{-\gamma} - \nu (\delta+s)^{-\gamma+\gamma_1}-D\right\} 
    \end{array} & {\rm in } \   \Omega\setminus \Omega_\delta\\[2ex]
     C_\delta( \delta+s)^{-\gamma} - \nu (\delta+s)^{-\gamma+\gamma_1}-D  & { \rm in } \ \Omega_\delta
     \end{array} \right.
     $$
   is the wanted sub solution.
   \smallskip
    
 {\bf 2. Uniqueness.}
      
We prove that $\underline{u}=\overline{u}$.

Remark that, by estimates \eqref{bouest}, for any $\theta<1$ and for any $c\in \R$,  
there exists $\delta$ such that $\theta \overline{u}(x) - c \leq \underline{u}(x)$ for  $d (x)\leq \delta$.  

 \medskip
   
{\em The case $\alpha\geq 0$.} Observe  that, for all $t\in \R$ and $c>0$, one has
$$
|t-c|^\alpha(t-c)-|t|^\alpha t\leq - 2^{-\alpha}c^{\alpha+1}\, .
$$
From this, we  deduce that
$$
-|\grad (\theta \overline{u}-c)|^{\alpha} F\left( D^2(\theta \overline{u}-c)\right) +|\grad (\theta \overline{u}-c)|^\beta +\lambda |\theta \overline{u}-c|^{\alpha} (\theta \overline{u}-c) \leq \theta^{\alpha+1} f(x) - \lambda 2^{-\alpha}  c^{\alpha+1}\, ,
$$
and the choice
$$
c=c_\theta =\left( \frac{2^\alpha (1-\theta^{\alpha+1}) | f^-|_\infty}{\lambda}\right)^{\frac{1}{\alpha+1}}
$$
then yields
$$
\begin{array}{r}
-|\grad (\theta \overline{u}-c)|^{\alpha} F\left( D^2(\theta \overline{u}-c)\right) +|\grad (\theta \overline{u}-c)|^\beta +\lambda |\theta \overline{u}-c|^{\alpha} (\theta \overline{u}-c)\\[2ex]
\qquad \leq \theta^{\alpha+1} f(x)- (1-\theta^{\alpha+1})\| f^-\|_\infty \leq f(x)\, .
\end{array}
$$
By applying Theorem \ref{comp}, it then follows that $\theta \overline{u}-c_\theta \leq \underline{u}$ in  $\Omega$, and letting $\theta \to 1$ we obtain the uniqueness of the explosive solution in the case $\alpha\geq 0$.

\medskip

{\em The case $\alpha < 0$. } 
In this case we use the inequality
$$
|t-c|^\alpha(t-c)-|t|^\alpha t\leq - 2^{\alpha}(\alpha+1) K^\alpha c\, ,
$$
which holds true for all $0<c<K$ and $t\in \R$ such that $|t|\leq K$. 

Let $C_1= \sup_\Omega |\overline{u}(x)|d(x)^\gamma$, which is finite by \eqref{bouest}. Then, for any $\delta>0$, one has $|\overline{u}|<\frac{C_1}{\delta^\gamma}$  in $\Omega_\delta$. Therefore, for any $0<\theta<1$ and $0<c<\frac{C_1}{\delta^\gamma}$, and for $x\in \Omega_\delta$, we have
$$
\begin{array}{c}
-|\grad (\theta \overline{u}-c)|^{\alpha} F\left( D^2(\theta \overline{u}-c)\right) +|\grad (\theta \overline{u}-c)|^\beta +\lambda |\theta \overline{u}-c|^{\alpha} (\theta \overline{u}-c)\\[2ex]
\quad  \leq \theta^{\alpha+1} f(x) - \lambda  (2C_1)^\alpha (\alpha+1)\delta^{-\alpha \gamma} c\, .
\end{array}
$$
We choose, as before,
$$
c=c_{\theta ,\delta}=\frac{|f^-|_\infty(1-\theta^{\alpha+1})}{\lambda (2C_1)^\alpha (\alpha+1)\delta^{-\alpha \gamma}}\, ,
$$
which is admissible for $\delta$ sufficiently small, since $\alpha>-1$. This yields
$$
-|\grad (\theta \overline{u}-c_{\theta ,\delta})|^{\alpha} F\left( D^2(\theta \overline{u}-c_{\theta ,\delta})\right) +|\grad (\theta \overline{u}-c_{\theta ,\delta})|^\beta +\lambda |\theta \overline{u}-c_{\theta ,\delta}|^{\alpha} (\theta \overline{u}-c_{\theta ,\delta})\leq f\quad \hbox{ in } \Omega_\delta .$$
On the other hand, by estimates \eqref{bouest} with $\nu=1$, we have
$$
\theta \overline{u}-c_{\theta ,\delta}\leq \underline{u}\quad \hbox{ on } \partial \Omega_\delta
$$
provided that
$$ 
\delta= \delta_\theta=\left( \frac{a (1-\theta)}{2(1+D)}\right)^{\frac{1}{\gamma_1}}\, .
$$
With this choice of $\delta$, we then deduce from Theorem \ref{comp} that $\theta \overline{u}-c_{\theta ,\delta_\theta}\leq  \underline{u}$ in $\Omega_{\delta_\theta}$.
Finally, we let $\theta\to 1$.  We observe that, 
by the restrictions  assumed on $\beta$ and $f$ in the case $\alpha<0$, we can choose $\gamma_1$ satisfying $\gamma_1>-\alpha \gamma$. Therefore,  $c_{\theta,\delta_\theta}\to 0$ as $\theta\to 1$, and we conclude that $\overline{u}\leq \underline{u}$ in $\Omega$.
           \end{proof}

\begin{rema}\label{new}
{\rm Let us put in evidence that estimates \eqref{bouest} imply that any solution $u$ of \eqref{eqlambdainfini} satisfies 
$$\lim_{d(x)\to 0} \frac{u(x)\, d(x)^\gamma}{C(x)}=1  \mbox{ if}\  \gamma>0, \quad \lim_{d(x)\to0}\frac{u(x)}{ |\log d(x)|\, C(x)}=1 \hbox{ if}\  \gamma=0\, .
$$
Moreover, if $f$ satisfies \eqref{f} with $\gamma_0=0$, then necessarily $\gamma_1=0$ and \eqref{bouest} reduce to
$$
\begin{array}{c}
(C(x)-\nu)d(x)^{-\gamma} -D\leq u(x)\leq (C(x)+\nu)d(x)^{-\gamma} +D\  \hbox{if } \gamma>0\, ,\\[1ex]
\left( C(x) - \nu\right)|\log d(x)| -D\leq u(x)\leq  \left( C(x) +\nu\right) |\log d(x)|+D
\ \hbox{ if } \gamma=0\, ,
\end{array}
$$
for any $\nu>0$, with $D>0$ depending in particular on $\nu$ and $\lambda$. The above estimates are the classical ones for explosive solutions firstly obtained in the semilinear case in \cite{LL}, where $C(x)$ is a constant function.  Also the case $\gamma_0=1$ has been considered in \cite{LL}, and in this case more refined estimates have been  obtained. In the nonlinear case, analogous estimates for $\gamma_0\geq 1$  would require further regularity assumptions on the non constant function $C(x)$.  Estimates \eqref{bouest}  are interesting in the intermediate cases $0\leq \gamma_0<1$, in which they are new also for linear operators and yield a uniqueness result in the non linear singular case $\alpha<0$. 
}\end{rema}

\section{A comparison principle  for non linear degenerate/singular equations without zero order terms}   \label{compzero}   
This section is devoted to some comparison principle for fully non linear equations  without zero order terms. For analogous results  concerning non singular operators,  see   \cite{BB}, \cite{BM}.               
  \begin{theo}\label{2.4}
Let $b$ be a continuous and bounded function in $\Omega$  and, when $\alpha \neq 0$, let  $f$ 
be a bounded continuous function such that $f \leq -m<0$.   Let $u$  and $v$ be respectively sub  
and super solution of 
   \begin{equation}\label{eq24}-|\nabla u |^\alpha F( D^2 u) +b(x) | \nabla u |^\beta = f\quad \mbox{in}\ \Omega.\end{equation}
If  $u$  and  $v$  are bounded and at least one of the two is Lipschitz continuous then the comparison principle holds i.e. 
$$ u \leq v\ \mbox{ on }\ \partial\Omega\Rightarrow u\leq v\ \mbox{ in }\ \Omega.$$ 
      \end{theo}
    
\begin{proof}Without loss of generality,
we will suppose that $u$ is Lipschitz continuous. 

The case $\alpha = 0$ is quite standard, it is enough to construct strict sub solutions that converge  uniformly to $u$. 
For $\kappa>0$   to be chosen, let
$u_\epsilon = u+ \epsilon e^{-\kappa x_1} -\epsilon$, with e.g.  $\Omega\subset \{x_1>0\}$. 
By the mean value theorem:  
  $$|\nabla u_\epsilon |^\beta  \leq
                                   |\nabla u |^\beta  + \beta  \kappa \epsilon e^{-\kappa x_1}\left(|\nabla u| + \epsilon k \right)^{\beta -1}.$$
 Since $\beta\leq 2$ in the case $\alpha=0$, we can  choose $\kappa$  such that $a \kappa > 2\beta  (|\nabla u |_\infty + \epsilon k)^{\beta -1}$.
  Then, $u_\epsilon$ is a 
strict sub solution of \eqref{eq24}, being
$$\begin{array}{rl}
F(D^2 u_\epsilon ) & \geq  F( D^2 u) + a \kappa^2 \epsilon  e^{-\kappa x_1} \\[2ex]
& \geq  F( D^2 u) +\beta k \epsilon  e^{-kx_1} (|\nabla u |_\infty + \epsilon k)^{\beta -1} + \frac{a \kappa ^2 \epsilon}{2}  e^{-\kappa  x_1}\\[2ex]
                   & \geq |\nabla u_\epsilon|^\beta -f + \frac{a \kappa ^2 \epsilon}{2}   e^{-\kappa  x_1}.
\end{array}
$$
 Furthermore $u_\epsilon\leq u\leq v$ on $\partial\Omega$.
We are now in a position to apply the comparison principle in Theorem \ref{comp}, and we obtain that 
$u_\epsilon \leq v$ in $\Omega$.  To conclude, let $\epsilon \to 0$. This computation has been done for  a classical solution $u$, but, with obvious changes, it can be made rigorous for viscosity solutions.
\medskip

For the case $\alpha \neq 0$ and $f<0$,
we use the change of function $u = \varphi (z)$, $v = \varphi (w)$ 
with $$ \varphi(s) = -\gamma ( \alpha+1) \log \left( \delta + e^{-\frac{s}{\alpha+1}}\right).$$
This function is  used in \cite{BB}, \cite{BM}, \cite{BP},  \cite{LP}, \cite{LPR}.

We choose $\delta$  
small enough in order that the range of $\varphi$  covers the ranges of $u$ and $v$.
The constant $\gamma$ will be chosen small enough depending only on $a$, $\alpha$, $\beta$, $\inf_\Omega (-f)$ and $|b|_\infty$; in this proof, any constant of this type will be called universal .  Observe 
that $\varphi^\prime>0$ while $\varphi^{\prime \prime} <0$. 
       
In the viscosity sense,   $z$ and $w$  are respectively sub and super solution of 
         \begin{equation}\label{eq1}-|\nabla z |^\alpha F( D^2 z +  \frac{\varphi^{\prime \prime}(z)}{\varphi^\prime (z)}\nabla z \otimes \nabla z)  +b(x) \varphi^\prime (z)^{ \beta-\alpha-1} | \nabla z |^\beta + \frac{-f}{(\varphi^\prime(z) )^{ \alpha+1}} = 0.
        \end{equation} 
 We define 
  $$H(x, s, p) =  \frac{-a\varphi^{\prime \prime}(s)}{\varphi^\prime (s)} |p|^{2+ \alpha} + b(x) \varphi^\prime (s)^{ \beta-\alpha-1} | p |^\beta + \frac{-f(x)}{\varphi^\prime(s) ^{ \alpha+1}}. $$
The point is to prove that  at $\bar x$, a maximum point  of $ z-w$, $\frac{ \partial H ( \bar x, s,p)}{\partial s} >0$ for all 
$p$. This will be sufficient to get a contradiction. 
A simple computation gives
           $$\varphi^\prime =\frac{ \gamma  e^{-\frac{s}{\alpha+1}}}{\delta + e^{-\frac{s}{\alpha+1}}},\ 
 \varphi^{\prime \prime} =\frac{ -\gamma\delta e^{-\frac{s}{\alpha+1}}}{(\alpha+1) ( \delta + e^{-\frac{s}{\alpha+1}})^2}.$$
Hence 
              $$\left(  \frac{ -\varphi^{\prime \prime}}{ \varphi^\prime } \right) ^\prime = \frac{\delta}{(\alpha +1)^2} 
              \frac{e^{-\frac{s}{\alpha+1}}}{(\delta + e^{-\frac{s}{\alpha+1}})^2}
   \ \mbox{i.e.}\  \left(  \frac{ -\varphi^{\prime \prime }}{\varphi^\prime } \right) ^\prime =  -\frac{\varphi^{ \prime \prime }}{( \alpha +1) \gamma}>0.$$
Differentiating $H$ with respect to $s$ gives: 
        $$ \partial _s H = a|p|^{ \alpha+2}  \frac{ -\varphi^{\prime \prime }}{( \alpha + 1) \gamma} 
        + (-f)\frac{ -(\alpha+1) \varphi^{\prime \prime}}{( \varphi^\prime)^{ \alpha+2}} +b(x)  |p|^\beta ( \beta-\alpha-1)( \varphi^\prime )^{ \beta-\alpha-2} \varphi^{\prime \prime}. $$
Since $-\varphi^{ \prime \prime}$  is positive, we need to prove that 
         
$$K:=\frac{ a |p|^{ \alpha+2}}{( \alpha+1) \gamma}+ (-f) {\frac{\alpha+1}{( \varphi^\prime )^{ \alpha+2} }} -|b|_\infty |p|^\beta ( \beta-\alpha-1) (\varphi^\prime )^{ \beta-\alpha-2}>0.$$
We start by treating the case $\beta<\alpha+2$.

Observe first that the boundedness of $u$ and $v$, implies that there exists universal  positive constants $c_o$ and $c_1$ such that 
$$\label{gammab}c_o\gamma\leq\varphi^\prime\leq c_1\gamma.
$$ 
Hence, it is easy to see that there exist three positive  universal constants $C_i$, $i=1,2,3$  such that
     $$K>\frac{ C_1 |p|^{ \alpha+2}}{\gamma}+\frac{C_2}{\gamma^{\alpha+2}}-\frac{C_3|p|^\beta}{\gamma^{\alpha+2-\beta}}.$$
We choose $\gamma=\min \left\{ 1, (\frac{C_3}{C_2})^\beta, (\frac{C_3}{C_1})^\frac{1}{\alpha+1-\beta}\right\}$.    With this choice of $\gamma$, 
for $|p|\leq 1$,
$$
\frac{ C_1 |p|^{ \alpha+2}}{\gamma}+\frac{C_2}{\gamma^{\alpha+2}}-\frac{C_3|p|^\beta}{\gamma^{\alpha+2-\beta}} \geq  \frac{C_2}{\gamma^{\alpha+2}}-\frac{C_3}{\gamma^{\alpha+2-\beta}} >0;
$$
while for $|p|\geq 1$,
$$\frac{ C_1 |p|^{ \alpha+2}}{\gamma}+\frac{C_2}{\gamma^{\alpha+2}}-\frac{C_3|p|^\beta}{\gamma^{\alpha+2-\beta}} \geq   \frac{(C_1) |p|^{ \alpha+2}}{\gamma}-\frac{C_3|p|^\beta}{\gamma^{\alpha+2-\beta}} >0.
$$
If $\beta = \alpha+2$, just take $\gamma < \frac{a}{( \alpha+1) | b|_\infty}$. 
               
This gives  that for $\gamma$ small enough depending only on $\min (-f)$ , $\alpha$, $|b|_\infty$  and $\beta$ one has, for some universal constant $C$,
\begin{equation}\label{gamma}
\partial_s H( x, s, p)\geq C>0.
\end{equation}
We now conclude the proof of the comparison principle.

We will distinguish the case $\alpha >0$ and $\alpha <0$. In the first case we introduce 
$ \psi_j(x, y) = z(x)-w(y)-\frac{j}{2} |x-y|^2$ while in the second case we use 
$ \psi_j(x, y) = z(x)-w(y)-\frac{j}{q} |x-y|^{q}$ where $q > \frac{ \alpha+2 }{\alpha+1}$. 
We detail the case $\alpha >0$. 
                   
Suppose by contradiction that $u> v$ somewhere in $\Omega$, then $z > w$ somewhere, since $\varphi$ is increasing,  while on the boundary $z \leq w$. Then the supremum of $z-w$ is positive and it is achieved inside $\Omega$. Hence $\psi_j$ reaches a positive maximum in $(x_j,y_j)\in\Omega\times\Omega$.

By Ishii's lemma \cite{I1},    there exists $(X_j, Y_j) \in  S\times S$ such that 
$$  (p_j, X_j)\in \overline{J}^{2,+} z(x_j),
                  \ (p_j, -Y_j)\in \overline{J}^{2,-} w(y_j),\ \mbox{ with}\ p_j=j(x_j-y_j) $$ 
 and 
$$\left( \begin{array}{cc} X_j & 0 \\0& Y_j\end{array} \right) \leq  j \left( \begin{array} {cc}I&-I\\-I&I\end{array}\right). $$
On $(x_j, y_j)$, by a continuity argument,  for $j$ large enough one has 
                $$ z(x_j) > w(y_j) +  \frac{\sup(z-w)}{2}.$$ 
Note for later purposes that since $z$ or $w$ are Lipschitz, $ p_j = j ( x_j-y_j)$ is bounded.
Observe that, the monotonicity of $\frac{ \varphi^{ \prime \prime }}{\varphi^\prime}$ implies that 
$$N =p_j \otimes p_j \left( \frac{ \varphi^{ \prime \prime } (z(x_j))}{\varphi^\prime ( z(x_j))}- \frac{ \varphi^{ \prime \prime }(w(y_j))}{\varphi^\prime(w(y_j))}\right)\leq 0.$$ 
Using the fact that $z$ and $w$ are respectively sub and super solutions of the equation  (\ref{eq1}), the estimate  \eqref{gamma} and that $H$ is decreasing in the second variable,  one obtains:  
                       \begin{eqnarray*}
     0& \geq & \frac{ -f(x_j)}{(\varphi^\prime )^{ \alpha+1} (z(x_j))} 
                        - |p_j|^\alpha F( X_j + \frac{ \varphi^{ \prime \prime } (z(x_j))}{\varphi^\prime ( z(x_j))}p_j \otimes p_j) + b(x_j)|p_j |^\beta \varphi^\prime ( z(x_j))^{\beta-\alpha-1} \\
     & \geq &\frac{ -f(x_j)}{(\varphi^\prime )^{ \alpha+1} (z(x_j))}  -|p_j |^\alpha F( -Y_j+  \frac{\varphi^{ \prime \prime } (w(y_j))}{\varphi^\prime ( w(y_j))} p_j \otimes p_j ) + \\
   &&+   a |p_j|^{2+ \alpha} \left( \frac{ \varphi^{ \prime \prime } (w(y_j))}{\varphi^\prime ( w(y_j))} -\frac{ \varphi^{ \prime \prime } (z(x_j))}{\varphi^\prime ( z(x_j))}\right)+  |p_j |^\beta  b(x_j)(\varphi^\prime ( z(x_j))^{ \beta-\alpha-1}\\
    & \geq & \frac{f(y_j) -f(x_j)}{( \varphi^\prime ( w(y_j))^{ \alpha+1} }+ (b(x_j)-b(y_j)) |p_j |^\beta ( \varphi^\prime ( w(y_j)))^{ \beta-\alpha-1}\\
&&+ H(x_j, z(x_j), p_j)- H(x_j, w(y_j),p_j)\\                        
 &\geq&C(z(x_j)-w(y_j))+  \frac{o(1)}{\gamma^{\alpha+1}}.          \end{eqnarray*}
Here we have used the continuity of $f$ and $b$, the boundedness of $p_j$ and  that 
$$\psi(x_j,y_j)\geq \sup (\psi(x_j,x_j), \psi(y_j, y_j)).$$ 
Passing to the limit one gets a contradiction, since $(x_j,y_j)$ converges to $(\bar x,\bar x)$ such that $z(\bar x)>w(\bar z)$.  This ends the case $\alpha\geq 0$.  
                       
In the case $\alpha<0$, the proof is similar but we need to make sure that one can choose $x_j\neq y_j$. 
This can be done proceeding  as in \cite{BDL1}.
\end{proof}

\section{Ergodic pairs}\label{sectionerg}
 In this section we consider, for $c\in \R$, the equation 
\begin{equation}\label{ergerg} -|\nabla u |^\alpha F(D^2 u) + |\nabla u |^\beta  = f + c\ \mbox{in}\ \Omega.\end{equation}
 \begin{defi}Suppose that $c$ is some constant (depending on $f$,  $\Omega$, $\beta$, $\alpha$,  and $F$)  such that there exists $\varphi  \in {\cal C}(\Omega)$, solution of \eqref{ergerg}, such that $\varphi \rightarrow +\infty$ at $\partial\Omega$. We will say that $c$ is an ergodic constant, $\varphi$ is an ergodic function and $(c, \varphi)$ is an ergodic pair. 
   \end{defi}
    We suppose, as usual,  that $\alpha>-1$, $\beta\in (\alpha+1,\alpha+ 2]$  and  recall that  
  $\gamma  = \frac{2+ \alpha -\beta}{\beta-\alpha-1}$ and $C(x)$ satisfies \eqref{C(x)}. In the following subsections, we prove the existence and show several properties of ergodic pairs.
  
 \subsection{Existence of   ergodic constants and boundary behavior of ergodic functions}
Theorem \ref{sympaetpas} provides the existence of a nonnegative ergodic  constant under the assumption that problem \eqref{eq3} does not have a solution. In the next result  we obtain the existence of ergodic constants using approximating explosive solutions.
     \begin{theo}\label{erginfi}
Let $F$ and $f$ be as in Theorem \ref{exilambdainf}, and assume further that $f$ is locally Lipschitz continuous in $\Omega$.
 Then, there exists an ergodic constant $c\in \R$.
  \end{theo}
\begin{proof}
By Theorem \ref{exilambdainf}, for $\lambda>0$ there exists a solution $U_\lambda$ of problem \eqref{eqlambdainfini}, which satisfies estimates \eqref{bouest}.  Recalling the dependence on $\lambda$ of the constant $D$ which appears in \eqref{bouest}, we see that $\lambda |U_\lambda|^\alpha U_\lambda$ is locally  bounded in $\Omega$, uniformly with respect to $0<\lambda<1$. Let us fix an arbitrary point $x_0\in \Omega$. Then,   there exists $c\in \R$ such that, up to a sequence $\lambda_n\to 0$, 
$$\lambda |U_\lambda (x_0)|^\alpha U_\lambda(x_0)\to -c\, .$$ 
On the other hand, Theorem \ref{lipsol} yields that $U_\lambda$ is locally uniformly Lipschitz continuous. Therefore, for $x$   in a compact subset of  $\Omega$, one has
$$
\lambda \left| |U_\lambda(x)|^\alpha U_\lambda(x)- |U_\lambda(x_0)|^\alpha U_\lambda(x_0)\right|\leq \lambda |U_\lambda (x)- U_\lambda(x_0)|^{\alpha+1}\to 0\quad \hbox{ if } \alpha\leq 0\, ,
$$
as well as, using again estimates \eqref{bouest},
$$
\lambda \left| |U_\lambda(x)|^\alpha U_\lambda(x)- |U_\lambda(x_0)|^\alpha U_\lambda(x_0)\right|\leq \lambda \frac{K}{\lambda^\frac{\alpha}{\alpha+1}}|U_\lambda (x)- U_\lambda(x_0)| \to 0\quad \hbox{ if } \alpha>0 \, .
$$
It then follows that $c$ does not depend on the choice of $x_0$ and, up to a sequence and locally uniformly in $\Omega$, one has
$$
\lambda |U_\lambda|^\alpha U_\lambda\to -c\, .
$$
Moreover, the function $V_\lambda(x)= U_\lambda (x)-U_\lambda(x_0)$ is locally uniformly bounded,  locally uniformly Lipschitz continuous  and satisfies
$$
-|\nabla V_\lambda|^\alpha F(D^2 V_\lambda)+|\nabla V_\lambda|^\beta=f-\lambda |U_\lambda|^\alpha U_\lambda\ \hbox{ in } \Omega\, .
$$
If $V$ denotes the local uniform limit of $V_\lambda$ for a sequence $\lambda_n\to 0$, then one has
$$
-|\nabla V|^\alpha F(D^2 V)+|\nabla V|^\beta=f+c\ \hbox{ in } \Omega\, .
$$
Finally, arguing as in the proof of Theorem \ref{sympaetpas} and using Theorem \ref{comp}, we have that, for some $\delta_0>0$ sufficiently small,
$$
V_\lambda \geq \phi +\min_{d(x)=\delta_0}V_\lambda\quad \hbox{ in }\Omega\setminus \Omega_{\delta_0}\, ,
$$
with $\phi$ defined in \eqref{phi} for arbitrary $s>0$. Letting $\lambda, s\to 0$   we deduce that $V(x)\to +\infty$ as $d(x)\to 0$. This shows that $(c,V)$ is an ergodic pair and concludes the proof.

\end{proof}

We now prove that ergodic functions satisfy on the boundary the same asymptotic identities as the explosive solutions of \eqref{eqlambdainfini}.
 \begin{theo}\label{esterg}
 Let $F$ and $f$ be as in Theorem \ref{exilambdainf}. Then, any ergodic function $u$ satisfies 
 \begin{equation}\label{bouasym}
\lim_{d(x)\to 0} \frac{u(x)\, d(x)^\gamma}{C(x)}=1  \mbox{ if}\  \gamma>0, \quad \lim_{d(x)\to0}\frac{u(x)}{ |\log d(x)|\, C(x)}=1 \hbox{ if}\  \gamma=0\, .
\end{equation}
\end{theo}
\begin{proof}
 As in the proof of Theorem \ref{exilambdainf}, we consider only  the case $\gamma>0$.
 
The computations made in the proof of Theorem \ref{exilambdainf} (for $\gamma_1=0$) show that, for all $\nu >0$  
and for $\delta_0>0$ sufficiently small,   the function
         $\overline{w}_{\nu, \delta}(x) :=\frac{C(x)+ \nu}{(d(x)-\delta)^\gamma}$ satisfies  for $\delta< d(x)<\delta_0$
         $$ -| \nabla  \overline{w}_{\nu, \delta} |^\alpha F(D^2 \overline{w}_{\nu, \delta})+ |\nabla \overline{w}_{\nu, \delta} |^\beta  \geq  c_1\, \nu 
         ( d(x)-\delta)^ {-\frac{\beta}{\beta-\alpha-1}}$$
where $c_1>0$ is a constant depending on $\alpha, \beta, a, A, |D^2d|_\infty$ and $|C|_{{\cal C}^2(\Omega)}$.  

By assumption \eqref{f} on $f$, this implies that
$$ -|\nabla  \overline{w}_{\nu, \delta}|^\alpha F(D^2 \overline{w}_{\nu, \delta})+ |\nabla \overline{w}_{\nu, \delta}|^\beta > f(x)+c=-|\nabla u|^\alpha F(D^2u)+|\nabla u|^\beta\ \hbox{ in } \Omega_\delta\setminus \Omega_{\delta_0}$$
for $\delta_0=\delta_0(\nu)$ small enough.
Hence, we are in the hypothesis   of  Theorem \ref{comp} and we deduce that
$$u \leq M_\nu + \overline{w}_{\nu, \delta}\ \hbox{ in } \Omega_\delta\setminus \Omega_{\delta_0},$$
with $M_\nu=\sup_{d(x) = \delta_0} u(x)$. 
Letting $\delta\to 0$  
we  obtain that 
$$  u   \leq M_\nu+ (C(x)+ \nu) d(x)^{-\gamma}  \mbox{ in }\ \Omega\setminus \Omega_{\delta_0}.$$
This in turn implies that
$$\lim_{d(x)\to 0} u(x)^{\alpha+1}d(x)^{\frac{\beta}{\beta-\alpha-1}-\gamma_0}=0$$
for all $\gamma_0$ such that
$$\gamma_0<\frac{\beta}{\beta-\alpha-1}-(\alpha+1) \gamma=\alpha +2.$$
Since $\alpha+2>1$, we obtain in particular that the function $|u|^\alpha u$ satisfies condition \eqref{f} with $\gamma_0=1$. Note also that $|u|^\alpha u$ is bounded from below in $\Omega$ since it is continuous and blows up  on the boundary.
Finally, we observe that $u$ satisfies
$$
-|\nabla u|^\alpha F(D^2u)+|\nabla u|^\beta +|u|^\alpha u=f+c+|u|^\alpha u\, ,$$
where the right hand side $f+c+|u|^\alpha u$ satisfies condition \eqref{f} with an exponent $\gamma_0= \min\{ \gamma_0(f), 1\}$, $\gamma_0(f)$ being the exponent appearing in the condition \eqref{f} satisfied by $f$. Hence, 
by applying Theorem \ref{exilambdainf},  we obtain that $u$ satisfies the boundary estimates \eqref{bouest} with $\lambda=1$ and the constant $D$ depending also on $u$ itself. Estimates \eqref{bouest} in turn imply relations \eqref{bouasym}.
 \end{proof}
 
\medskip          

\subsection{Uniqueness and further properties of the ergodic constant: proof of Theorem \ref{ABCDE}.}\label{properg}

Throughout this section we assume that $f$ is bounded and locally Lipschitz continuous.

In the introduction we have defined $\mu^\star \in \R\cup \{-\infty\}$ as
$$\mu^\star =  \inf  \{ \mu \, :\,  \exists \, \varphi \in { \cal C} ( \overline{ \Omega}), -| \nabla \varphi |^\alpha F( D^2 \varphi) + | \nabla \varphi |^\beta  \leq f+ \mu\}.$$
It is easy to see that  $\mu^\star \leq -\inf_\Omega f$. A better upper bound on $\mu^\star$ depending on the domain $\Omega$ is given by the following result.
\begin{prop}\label{lowerupperbounds} 
  If $\Omega\subset [0, R] \times \R^{N-1}$, then 
             $$ \mu^\star \leq -\inf_\Omega f - \frac{ K_1}{R^{ \frac{\beta}{\beta-\alpha-1}}}$$
for a  positive constant   $K_1=K_1(a,\alpha,\beta)$.
\end{prop}
\begin{proof} The function $\varphi (x) = C x_1^{\frac{\alpha+2}{\alpha+1}}$ with $C = \left[\frac{a}{2(\alpha+1)}\right]^{\frac{1}{\beta-\alpha-1}}\left(\frac{\alpha+1}{\alpha+2}\right)R^{ -\frac{\beta}{( \alpha+1) ( \beta-\alpha-1)}}$
satisfies, for some constant $K_1=K_1(a, \alpha, \beta)$:
 $$- |\grad \varphi |^\alpha F(D^2\varphi)+|\grad \varphi |^\beta\leq-\frac{a(\alpha+2)^{1+ \alpha}}{( \alpha+1)^{2+ \alpha} }C^{1+ \alpha} + \frac{ ( \alpha+2)^{\beta }}{(\alpha+1)^\beta} C^\beta x_1^{ \frac{\beta}{\alpha+1}}
 \leq  - K_1R^{-\frac{\beta}{\beta-\alpha-1}}.$$
Hence, by its definition, $\mu^\star \leq -\inf_\Omega f- K_1R^{-\frac{\beta}{\beta-\alpha-1}}$. 
\end{proof}
We are now ready to give the proof of Theorem \ref{ABCDE}.        
\begin{proof}[Proof of Theorem \ref{ABCDE}] Here we set $c_\Omega=c$.

 The existence of $c$ is given by Theorem \ref{erginfi}.
\smallskip

{\em Proof of } \ref{A}.   Suppose that $c$ and $c^\prime$ are two ergodic constants, and let $\varphi$ and $\varphi ^\prime$ be respectively   corresponding ergodic functions. 
By Theorem \ref{esterg} the ratio of $\varphi$ and $\varphi^\prime$ goes to 1 as $d(x)\to 0$; 
hence, for any $\theta <1$, the supremum of $\theta \varphi-\varphi^\prime $ is achieved in the interior of $\Omega$ since
  $\theta \varphi-\varphi^\prime $ blows down to $-\infty$  as $d(x)\to 0$.
  
We observe that 
$$ \begin{array}{c}
- |\nabla ( \theta \varphi)|^\alpha F( D^2 ( \theta \varphi)) + |\theta \nabla \varphi |^\beta  \leq  \theta^{1+ \alpha} (f+c)\\[1ex]
-|\nabla \varphi'|^\alpha F(D^2 \varphi')+ |\nabla \varphi'|^\beta=f+c
\end{array}
$$ 
From standard  comparison arguments in  viscosity solutions theory, see \cite{usr},  it follows that at  
 a maximum point $\bar x$ of $\theta \varphi-\varphi^\prime $ one has
                  $f(\bar x)+c^\prime \leq \theta^{1+ \alpha} ( f(\bar x)+c).$
Letting $\theta  \to 1$, we get $ c' \leq c$. Exchanging the roles of $c$ and $c^\prime$  we conclude that $c=c'$. 
\smallskip

{\em Proof of } \ref{B}.   
Let $\mu<c$ and suppose by contradiction that there exists $\varphi   \in { \cal C} ( \overline{ \Omega})$  satisfying 
                  $$ -|\nabla\varphi |^\alpha F(D^2 \varphi) + |\nabla \varphi|^\beta \leq f+ \mu .$$
Let $u$ be an  ergodic function corresponding to $c$. Clearly,  $\sup_\Omega (\varphi- u)$ is attained in $\Omega$. 
Again by standard viscosity arguments, we obtain that  at   a maximum point $\bar x$ of $\varphi-u$ one has
                     $$ f(\bar x)+\mu \geq f(\bar x) + c\, , $$
which is a contradiction. Hence, we deduce
$$
\{ \mu \, :\,  \exists \, \varphi \in { \cal C} ( \overline{ \Omega}), -| \nabla \varphi |^\alpha F( D^2 \varphi) + | \nabla \varphi |^\beta  \leq f+ \mu\}\subset [c, +\infty)\, ,
$$
which implies that $\mu^\star$ is finite and $\mu^\star\geq c$.

On the other hand,  by definition of $\mu^\star$, for  any $\mu < \mu^\star$ the problem 
  $$\left\{ \begin{array}{cl}
                     -| \nabla u |^\alpha F(D^2 u) + |\nabla u |^\beta    = f + \mu & { \rm in} \ \Omega \\
                     u=0  & { \rm on} \ \partial \Omega
                     \end{array}\right.$$
does not have solution.
                       Theorem \ref{sympaetpas} then  implies that there exists an ergodic constant $c_{ f+ \mu}\geq 0$ for the right hand side  $f+ \mu$. On the other hand, by the uniqueness proved in \ref{A}. above, one has  $c=\mu+c_{f+\mu}$. Hence,  $c\geq \mu$ and, therefore, $c\geq \mu^\star$.
\medskip  
             
{\em Proof of } \ref{C}.
The nondecreasing monotonicity of $c$ with respect  to the domain $\Omega$ is an immediate consequence of point \ref{B}.  above and the definition of $\mu^\star$.

Let us now prove the "continuity" of $\Omega \mapsto c_\Omega$, in the  following weak sense. For $\delta>0$ small, let $c_\delta$ denote the 
 ergodic constant in $\Omega_\delta$.  Then, $c_\delta$ is nondecreasing as 
$\delta$ decreases to zero, and $c_\delta \leq  c=c_\Omega$. 

Let $u_\delta$ be an ergodic function in $\Omega_\delta$,   $x_0\in \Omega$ be a fixed point and let us set $v_\delta=u_\delta-u_\delta(x_0)$.
            
By  Theorem \ref{lipsol},  $v_\delta$ is locally uniformly bounded and locally
uniformly Lipschitz continuous in $\Omega_\delta$. Thus, up to a sequence $\delta_n\to 0$, $v_\delta$ converges locally uniformly in $\Omega$ to a solution $v$ of the equation with   right hand side $f+\lim_{\delta\to 0} c_\delta$. Moreover,  arguing as in the proof of  Theorem \ref{erginfi}, we have that
$v _\delta(x) \geq C_0 \left( d(x)-\delta\right)^{-\gamma}$ if $\gamma>0$, and  $v _\delta(x) \geq -C_0 \log \left( d(x)-\delta\right)$ if $\gamma=0$, for some  constant $C_0>0$ and for  $\delta<d(x)\leq \delta_0$. Letting $\delta\to 0$, we get that $v(x)\to +\infty$ as $d(x)\to 0$. Hence, $v$ is an ergodic function in $\Omega$ and, by point \ref{A}.,     $\lim_{\delta\to 0} c_\delta$ is the ergodic constant $c$.
 \medskip
       
{\em Proof of} \ref{D}. 

We prove that the constant $\mu^\star$ is not achieved. 
Suppose by contradiction that there exists $\varphi\in {\cal C} (\overline{ \Omega})$ such that 
                     $$-| \nabla \varphi |^\alpha F( D^2 \varphi) + |\nabla \varphi|^\beta  \leq   f+ \mu^\star=f+c \,. $$
On the other hand, let  $u$ be an ergodic function in $\Omega$.  

We observe that for all constants $M$, $\varphi+M$ is still a bounded sub solution, whereas $u$ is a solution satisfying $u=+\infty$ on $\partial \Omega$. Theorem \ref{2.4} applied in a smaller domain $\Omega_\delta$ then yields $u\geq \varphi +M$ for arbitrarily large $M$, which clearly is a contradiction.               
 
 A similar argument proves the strict increasing behavior of the ergodic constant. Let $\Omega' \subset \subset \Omega$  and      
        suppose by contradiction that $c_{\Omega'} = c_{ \Omega}$.  Let  $u_{ \Omega^\prime}$ and $u_\Omega$ be  ergodic functions respectively in $\Omega'$ and $\Omega$.  For every constant $M$, both $u_{\Omega}+M$ and $u_{\Omega'}$ satisfy \eqref{ergerg} in $\Omega'$, with $u_\Omega+M$ bounded and $u_{\Omega'}=+\infty$ on $\partial \Omega'$. Hence, Theorem \ref{2.4} yields the contradiction $u_{\Omega'}\geq u_{\Omega}+M$ for every $M$.
 \end{proof}

\begin{rema}
{\rm We remark that, thanks to Proposition \ref{lowerupperbounds}, the condition $\sup_\Omega f+c<0$ appearing in Theorem \ref{ABCDE}-\ref{D}. is satisfied in one of the following cases:
\begin{itemize}
\item[--] $f$ is constant in $\Omega$;

\item[--] the oscillation $\sup_\Omega f-\inf_\Omega f$ of $f$ is suitably small, in dependence of the length of the projections of $\Omega$ on the coordinated  axes;

\item[--] in at least one direction $\Omega$ is suitably narrow, in dependence of the oscillation  of $f$ in $\Omega$.
\end{itemize}} 
\end{rema}
\bigskip

{\bf  Acknowledgment}. 
Part of this work  has been done while  the first and third authors were visiting   the UMR 80-88, University of Cergy Pontoise,  and the second one was visiting  Sapienza University of Rome supported by  GNAMPA- INDAM
.

          \end{document}